\documentclass[11pt]{amsart}
\usepackage[a4paper, left=2.5cm, right=2.5cm, height=23cm]{geometry}
\usepackage{graphicx}
\usepackage{ulem}
\usepackage{chngcntr}
\counterwithin{figure}{section}
\usepackage{amsfonts,mathrsfs}
\usepackage{amssymb}
\usepackage{amsmath}
\usepackage{amsthm}
\usepackage{mathtools}
\usepackage{latexsym}
\usepackage[english]{babel}
\usepackage{inputenc}
\usepackage{color}
\usepackage{subcaption}
\usepackage{enumerate}
\usepackage{dsfont}
\usepackage{upgreek}
\usepackage{graphicx}
\usepackage{mathrsfs}
\usepackage{marginnote}
\usepackage{mathtools}
\usepackage{amsthm,thmtools,xcolor}

\usepackage[all]{xy}
\usepackage{tikz-cd}
\usepackage{bigints}
\usepackage[square,sort,comma,numbers,nonamebreak]{natbib}
\theoremstyle{plain}
\numberwithin{equation}{section}
\newtheorem{thm}{Theorem}[section]

\newtheorem{lem}[thm]{Lemma}
\newtheorem{pro}[thm]{Proposition}

\newtheorem{claim}{Claim}

\newtheorem*{thm*}{Theorem}
\newtheorem{defi}[thm]{Definition}
\theoremstyle{remark}
\newtheorem{remark}[thm]{Remark}
\usepackage{color}
\usepackage[colorlinks = true,
linkcolor = blue,
urlcolor  = blue,
citecolor = blue,
anchorcolor = blue]{hyperref}
\usepackage{hyperref}
\makeatletter
\DeclareMathOperator{\Ric}{Ric}
\DeclareMathOperator{\R}{R}

\title[On the existence of $k$-Yamabe gradient solitons]{On the existence and classification of $k$-Yamabe gradient solitons}

\author[M.F. Espinal]{Maria Fernanda Espinal$^{*}$}
\address{Facultad de Matem\'aticas, Pontificia Universidad Cat\'olica de Chile,  Avenida Vicuña Mackenna 4860, Santiago, Chile.}
\email{mfespinal@uc.cl}

\author[M. S\'aez]{Mariel S\'aez}
\address{Facultad de Matem\'aticas, Pontificia Universidad Cat\'olica de Chile,  Avenida Vicuña Mackenna 4860, Santiago, Chile.}
\email{mariel@uc.cl}
\thanks{$^{*}$ M.F.E.'s work was supported by the Agencia Nacional de Investigación y Desarrollo ANID Gobierno de Chile, fellowship number 21190289.
}
\date{\today}

\begin{document}

%%%%%%%%%%%%%%%%%%%%%%%%%%%%%%
	\begin{abstract}
		In this paper we classify rotationally symmetric conformally flat admissible solitons to the $k$-Yamabe flow, a fully non-linear version of the Yamabe flow. For $n\geq 2k$ we prove existence of complete expanding, steady and shrinking solitons and describe their asymptotic behavior at infinity. For $n<2k$ we prove that steady and expanding solitons are not admissible. The proof is based on the careful analysis of an associated dynamical system.
	\end{abstract}
	
	\maketitle
	
	%	\tableofcontents
	
	\begin{center}
	\end{center}
	
	\bigskip
	
	\centerline{\bf AMS subject classification:  53C18, 53C21, 35C08
		58J60, 34A34}
	
	\section{Introduction and main results}

	Let $(M, g)$ be a complete, connected smooth Riemannian manifold of dimension $n \geq 3$. Let $\Ric_{g}$, $\R_{g}$ be the Ricci tensor and scalar curvature of $g$, respectively. The Schouten tensor with respect to the metric $g$ is given by
	\begin{equation}A_{g}=\frac{1}{n-2}\left(\Ric_{g} -\frac{\R_{g}}{2(n-1)}g \right). \label{def Schouten}\end{equation}
	We are interested in  considering the following curvature flow
	\begin{align}
		\label{k-flow}
		\begin{cases}
			\frac{d}{dt}g&=-\sigma^{1/k}_{k}(g)g\\
			g(0)&=g_{0},
		\end{cases}
	\end{align}
	with
	\begin{equation}
		\label{conformal_flow_metric}
		g(\cdot, t)= \overline{u}^{\frac{4k}{n+2k}}(\cdot, t) \,g_0(\cdot) \quad \text{and} \quad u>0.
	\end{equation}
	Here $\sigma_{k}(g)$ denotes $k$-th symmetric function of the eigenvalues of the $(1, 1)$-tensor $g^{-1}A_{g}$
	\begin{align}
		\label{sigma_definition}
		\sigma_{k}(g) := \sigma_{k}(g^{-1}A_{g})= \sum_{i_{1}<i_{2}<...<i_{k}}\lambda_{i_{1}}\ldots\lambda_{i_{k}}, \quad
		\text{for} \quad 1 \leq k \leq n, 
	\end{align}
	where $\lambda_1, \ldots, \lambda_n$ are the eigenvalues of $g^{-1}A_{g}$.\\
	
	For $k=1$ we have $\sigma_1(g)=\frac{\R_{g}}{2(n-1)}$  and \eqref{k-flow} agrees with the 
classical Yamabe flow, which can be seen as the  parabolic version of the well-known Yamabe problem. That problem  seeks to prove existence of metrics of constant scalar curvature within a conformal class and  was settled in 1984 as the conclusion  of several works by H. Yamabe, N. Trudinger, T. Aubin and R. Schoen (see \cite{Parker-Li, Marques} for surveys in this subject and  precise references). The problem can be reduced to solving a semi-linear elliptic PDE with critical exponent and that criticality  imposes several difficulties in the analysis. In \cite{hamilton1989lectures}, R. Hamilton proposed the Yamabe flow as an alternative perspective, conjecturing that solutions to the parabolic equation would asymptotically  approach (as $t\to \infty$) the desired metrics. For closed manifolds a positive answer to Hamilton's conjecture was obtained by the works of B. Chow  \cite{chow1992yamabe}, R. Ye \cite{ye1994global} and S. Brendle 
 \cite{brendle2007convergence,brendle2005convergence}. \\
 
 For non-compact manifolds Yamabe's conjecture does not hold in full generality (see \cite{jin1988counterexample}), but there are a many works seeking for optimal conditions (there are many references, a non comprehensive list are for instance  \cite{AvilesMcOwen, schoen1988conformally, bettiol2016bifurcation, akutagawamondello}
  for the elliptic problem and \cite{Choi_Daskalopoulos, schulz2020noncompact} in the parabolic case); however, there are still many open questions for $M$  non-compact. \\ 	
	
	The $k$-Yamabe problem consists in finding metrics of constant $\sigma_k$-curvature and one motivation is that these curvature quantities bear a stronger connection with the underlying topology (than the scalar curvature).
	One example of this statement is the Chern-Gauss-Bonnet formula in 4 dimensions that is given by
\begin{equation*}
	%\label{eq:Chern-Gauss-Bonnet}
	8\pi^{2}\chi(M)=\int_{M}\left(\frac{1}{4}|W_{g}|^{2}+\sigma_{2}(g)\right)dv_{g}.
\end{equation*}

For $k\geq 2$ the equation 
		$$\sigma_k(g)=K$$
	is fully non-linear and additional assumptions are necessary to guarantee ellipticity. A standard condition is to consider metrics within the positive cone
	\begin{align}
		\label{cono_definition_g}
		\Gamma_{k}^{+}=\{g:\sigma_{1}(g),\ldots,\sigma_{k}(g)>0\}.
	\end{align}
Under this condition, several authors \cite{viaclovsky2000conformal, chang2002equation,chang2002priori,li2003some, gursky2007prescribing} have studied the $k$-Yamabe problem  for $k\geq 2$ 
 and  there are also a few results for $\sigma_k<0$ with different assumptions (that also ensure ellipticity of the equation),
 see for instance
\cite{duncan2024fully} and references therein. From the perspective of geometric flows, a related equation was studied in the positive cone $\Gamma_{k}^{+}$ by  P. Guan and G. Wang \cite{guan2003fully}; they considered a
 manifold $M$ compact, locally conformally flat and   $2k\ne n$. These conditions implied long time convergence to the desired metrics. The result in \cite{guan2003fully} was later extended  in \cite{Weimin_negative} for
 $\sigma_k<0$ (also assuming conditions that ensure ellipticity).\\

A relevant observation is that for $n=2k$ the integral 
\begin{equation*}
	\int_M \sigma_{k}(g)dv_{g}
\end{equation*}
is a conformal invariant and, in fact, for several questions related to the $k$-Yamabe problem there are different behaviors
 depending on whether $2k<n$,  $2k=n$  or  $2k>n$. We will also observe differences in our results, depending on the sign of $n-2k$.\\

In this paper we focus on {\it soliton} solutions to \eqref{k-flow}. More precisely, a solution $g(t)$, is called a  soliton if there exists a smooth function $\tau(t)$ and $1$-parameter family of diffeomorphisms $\{\phi_{t}\}$ of $M$ such that
	\begin{equation*}\label{def:soliton} g(t)=\tau(t)\phi_{t}^{*}(g_{0}),\end{equation*} with $\tau(0)=1$ and $\phi_{0}=id_{M}$. From this point onward, for the sake of notational simplicity, we will write $g$ instead of $g_{0}$.
	 Equation \eqref{k-flow} reduces to the elliptic problem
	\begin{align*}
		%\label{soliton_equation}
		(\sigma^{1/k}_{k}(g)-\rho)g =\frac{1}{2}\mathcal{L}_{X}g,
	\end{align*}	
	where $\rho=-\tau'(0)$, $X$ is the vector field generated by the $1$-parameter family $\phi_{t}$ and $\mathcal{L}_{X}$ denotes the Lie derivative of $X$.  If $X$ is a gradient vector field, i.e. $X=\nabla \varphi$ for a function $\varphi$, then $(M,g)$ is called a \textit{gradient soliton}. In that setting,   $\frac{1}{2}\mathcal{L}_{X}g=\nabla^{2} \varphi$ and the previous equation becomes
	\begin{equation}
		\label{soliton_f}
		(\sigma^{1/k}_{k}(g)-\rho)g =\nabla^{2} \varphi.
	\end{equation}
%When $\varphi$ is a constant function, we call it a trivial soliton. Moreover,	
	A soliton is called shrinking, steady or expanding if $\rho>0$, $\rho=0$, or $\rho<0$, respectively.\\
	
	The study of soliton solutions to geometric flows is interesting  as special examples of solutions to these equations and also because they often have deep  connections with the formation of singularities (see for instance \cite{Choi_Daskalopoulos} for the Yamabe flow and \cite{EndersMullerTopping} for the Ricci flow). 
In particular,  classification results for soliton solutions  to  geometric flows are expected to contribute to a better understanding of possible singular behavior that may develop along the evolution.\\

 In the case of the Yamabe flow ($k=1$)	
the classification of conformally flat rotationally symmetric solitons  was achieved by
 work of P. Daskalopoulos and N. Sesum in \cite{daskalopoulos2013classification}. 
 The assumption of rotational symmetry in that situation is justified by another result in \cite{daskalopoulos2013classification} that states
  that if the  sectional curvature is positive, then locally conformally flat complete Yamabe gradient solitons are rotationally symmetric. That result was extended in \cite{catino2012global} for other conformal gradient solitons (including our case), where the authors show that any complete, noncompact $k$-Yamabe gradient soliton   $(M^{n},g)$ with nonnegative Ricci tensor is either a direct product $\mathbb{R}\times N^{n-1}$ where $(N^{n-1},g_{N})$ is an $(n-1)$-dimensional complete Riemannian manifold with nonnegative Ricci tensor, or $(M^{n},g)$ is rotationally symmetric and globally conformally equivalent to $\mathbb{R}^{n}$  (\textit{Theorem 3.6} in \cite{catino2012global}).\\

 	The aim of our work is to provide classification results for  solutions to \eqref{soliton_f} that are rotationally symmetric. With this goal in mind, 
	our first result reduces the classification of conformally flat rotationally symmetric $k$-Yamabe gradient solitons to the classification of global smooth solutions of a fully nonlinear elliptic equation.\\

\begin{thm}[PDE formulation of $k$-Yamabe gradient solitons] \label{PDE_formulation}
		\label{PDE_formulation}
		Let $g_{u}$ defined by 
		\begin{equation*}
			\label{conformal_g}
			g_{u}(\cdot)=u^{\frac{4k}{n+2k}}(\cdot)|dx|^{2}, \quad u>0.
		\end{equation*}
	Then  $g_{u}$ is a conformally flat rotationally symmetric $k$-Yamabe gradient soliton with $\sigma_{k}(g_{u})>0$ if and only if $u$ is a smooth radial solution to the elliptic equation
		\begin{equation}
			\begin{split}
				\label{elliptic_equation}
				\sigma_k^{1/k}\left(D^{2}(u^{-\frac{2k}{n+2k}})-\frac{u^{-\frac{2k}{n+2k}}}{2}|D(u^{-\frac{2k}{n+2k}})|^{2}I\right) =u^{\frac{2k}{n+2k}-1} \left[(2\theta +\rho) u +\frac{4k\theta}{n+2k}\, x\cdot Du \right],
			\end{split}
		\end{equation}
		%$\theta=\frac{C}{4(n-1)}$, and $C=\frac{f_{s}}{r^{2}u^{\frac{4k}{n+2k}}}$. 
		where $\theta$ is a parameter that satisfies $2\theta+\rho>0$. Equation \eqref{elliptic_equation}
		equivalent to
		\begin{align*}D_j(D_i u^{-\frac{2k}{n+2k}} T^{ij}_{k-1})&-n u^{\frac{2k}{n+2k}}T^{ij}_{k-1}D_i u^{-\frac{2k}{n+2k}}D_j u^{-\frac{2k}{n+2k}}\\ &+u^{\frac{2k}{n+2k}}\frac{n-k+1}{2}\sigma_{k-1}|Du^{-\frac{2k}{n+2k}}|^{2}\\&=ku^{\frac{k(n-2k)}{n+2k}} \left[(2\theta +\rho) u +  \frac{4k\theta}{n+2k}  x\cdot Du\right]^k,
		\end{align*}
		where $T_{k-1}$ is the $(k-1)$-th Newton tensor evaluated on the Schouten tensor $g_{u}^{-1}A_{g_u}$.
		All derivatives and norms are Euclidean.
	\end{thm}

	Theorem \ref{PDE_formulation}
	is obtained in Section \ref{sec:pdeformulation} by a direct computation after imposing radial symmetry. Moreover,
		expression \eqref{elliptic_equation} allow us to extend the theory developed by Vázquez in \cite{vazquez2006smoothing} (and used for the case $k=1$). We obtain the following result.	
			\begin{thm}[Existence of radial $k$-Yamabe gradient solitons for $n\geq 2k$]
		\label{classification}
		Let  $n\geq 2k$. For every $\alpha>0$ the elliptic equation \eqref{elliptic_equation} admits non-trivial admissible radially symmetric smooth solutions $u_\alpha$ that satisfies $u_\alpha(0)=\alpha$    if and only if 
		$$\theta>0 \quad \text{and} \quad 2\theta +\rho>0.$$

		\end{thm}
	
	We remark that a solution $u$ is admissible if the associated metric $g_u$ belongs to the positive cone $\Gamma_{k}^{+}$.\\

	 The solutions of Theorem \ref{classification} can be distinguished among each other from their asymptotic behavior at infinity, which depends on the values of $\rho$, $\theta$, $n$ and $k$.
	These behaviors are summarized by the following theorem.

	\begin{thm}[Asymptotic behavior]
		\label{classification2}
		Let $m=\frac{n-2k}{n+2k}$, $n\geq 2k$ and $\rho$, $\theta$ as in Theorems \ref{PDE_formulation} and \ref{classification}. Let $u_\alpha$ be the solutions given by Theorem \ref{classification}, then the following holds.
		
	\medskip
		
		\begin{enumerate}
			\item \textbf{Yamabe expander $\rho<0$:}  It holds $u_\alpha (x)= \text{O}(|x|^{-2-\delta})$ as $|x|\to \infty$, where $\delta=\frac{\rho}{\theta (1-m)}$.\\

			\item \textbf{Yamabe steady $\rho=0$:}\begin{itemize}\item For $n>2k$ the decay rate at infinity is given by  $$u_\alpha (x)=\text{O}\left(\left[\frac{\ln |x|}{|x|^2}\right]^{\frac{1}{1-m}}\right)\hbox{ as }|x|\to \infty.$$ 
			\item If $n=2k$ we have $$u_{\alpha}(x)= \text{O}\left(\frac{\,(\ln |x|)^{1-\frac{2}{n}}}{|x|^2} \right) \hbox{ as }|x|\to \infty,$$\end{itemize}
			\item \textbf{Yamabe shrinker $\rho>0$:}
			
			\begin{itemize}
			 \item If $n>2k$ and $0<\rho\leq 2\theta$ then solutions have a slow-decay rate at infinity, namely $$u_{\alpha}(x)=\text{O}(|x|^{-\frac{2}{1-m}}) \hbox{ as }|x|\to\infty.$$
			\item  If $n>2k$ and $0< 2\theta<\rho$ solutions either  have a slow-decay rate at infinity \begin{align*}u_{\alpha}(x)&=\text{O}(|x|^{-\frac{2}{1-m}}) \hbox{ as }|x|\to\infty \hbox{ or } \\ u_{\alpha}(x)&=\text{O}(|x|^{-\frac{4}{1-m}}) \hbox{ as } |x|\to\infty.\end{align*}
			\item If $n=2k$ then  $$u_{\alpha}(x)=\text{O}(|x|^{-2(1+d)})  \hbox{ as }|x|\to\infty$$ for some
			 $0<d\leq  \min\{\frac{\rho}{2\theta},1\}$.					
			\end{itemize}
		\end{enumerate}
	\end{thm}
%We remark that for $\rho> 2\theta$ we do not know the precise behavior of the solution (which decay occurs if $n>2k$, nor which values of $d$ are possible when $n=2k$). \\
We remark that for $\rho\geq 2 \theta$ our analysis is not detailed enough to determine for which values of the parameters  $\rho$ and $\theta$ the decay is slow (nor  for which ones is fast). Moreover, when $n=2k$ remains open to determine whether is possible to find smooth solutions with decay $O(|x|^{-2(1+d)})$ for every $d\in (0,  \min\{\frac{\rho}{2\theta},1\}]$. \\

Finally,   when $n<2k$ we can also partially perform the analysis of solutions, obtaining the following result.
	\begin{thm}[Non-existence of radial $k$-Yamabe gradient solitons for $n<2k$]
		\label{classification3}
		Let $n<2k$ and \\ $m=\frac{n-2k}{n+2k}<0$, then  the elliptic equation \eqref{elliptic_equation} does not have an admissible solution for $\rho< 2\theta$. \\
		
		For $\rho\geq 2 \theta$ and $\alpha>0$,  there exists a one parameter family $u_{\alpha}$ of smooth radially symmetric admissible solutions of Equation \eqref{elliptic_equation} on $\mathbb{R}^{n}$ and their decay rate is  given   by  $u_{\alpha}(x)=\text{O}(|x|^{-\frac{4}{1-m}})$ as $|x|\to\infty$.			\end{thm}

\medskip

		We observe that solutions to the curvature flow \eqref{k-flow} can be recovered from  a solution $u$ to  \eqref{elliptic_equation} as follows:
		
		\begin{enumerate}
			\item[i.] $k$-Yamabe shrinkers $\rho>0$: $$\overline{u}(x,t)=(T-t)^{\beta\gamma}u(\eta), \quad \eta=|x|(T-t)^{\beta}.$$
			\item[ii.] $k$-Yamabe expander $\rho<0$: $$\overline{u}(x,t)=t^{-\beta\gamma}u(\eta), \quad \eta=|x|t^{-\beta}.$$
			\item[iii.] $k$-Yamabe steady $\rho=0$: $$\overline{u}(x,t)=e^{-\beta\gamma t}u(\eta), \quad \eta=|x|e^{-\beta t}.$$
		\end{enumerate}	
		In all of the above cases $\beta=(1-m)\theta$, $\gamma=\frac{2\theta+\rho}{\beta}$, the function $u$ is solution of the Equation \eqref{elliptic_equation} and $\overline{g}(\cdot, t)=\overline{u}^{\frac{4k}{n+2k}}(\cdot,t)|dx|^2$ defines a solution of the $k$-Yamabe flow \eqref{k-flow}.\\

		% Theorem \ref{classification}  builds on extending the theory developed by Vásquez in \cite{vazquez2006smoothing}; more precisely, it reduces the problem to a dynamical system with $3$ critical points. The key step is to prove the existence of appropriate orbits in this dynamical system. The asymptotic behavior is computed directly from the equations of the associated system.
	
\subsubsection*{Organization of this paper}
	This paper is organized as follows. In Section \ref{sec:preliminaries}  we give the necessary background
	on $\sigma_{k}$ and its properties. The PDE for the $k$-Yamabe gradient soliton (Theorem \ref{PDE_formulation}), along with its formulation in terms of the eigenvalues of the Schouten tensor and the criteria for the admissibility of the solution are described in Section \ref{sec:pdeformulation}. In Section \ref{sec:ODEanalysis}, we derive an autonomous system of ordinary differential equations, that is equivalent to \eqref{elliptic_equation}. In that section we also
identify the critical points of the system, describe the possible orbits and establish the existence of solutions near the origin. The existence and non-existence of radial $k$-Yamabe gradient solitons (Theorems \ref{classification} and \ref{classification3}) are proved in Section \ref{sec:globalexistence}. The asymptotic behavior (Theorem \ref{classification2}) is studied in Section \ref{sec:asymtoticbehavior}.

	\subsection*{Acknowledgments.} This work is part of the doctoral dissertation of the first author at Pontificia Universidad Cat\'olica de Chile, under the guidance of  the second author. The authors also wants to express their gratitude to  M.~d.~M. Gonz{\'a}lez, for her interest, time, and many helpful suggestions.\\

	\section{Preliminaries}
	\label{sec:preliminaries}
	
	In this section, we collect known properties of $\sigma_k$ and the Schouten tensor. Throughout the paper, we use Einstein's summation convention on repeated indices.\\

	Consider a real valued $n\times n$ matrix $\mathds{A}\in \mathcal{M}_{n\times n}(\mathbb{R})$. The elementary symmetric functions $\sigma_k(\mathds{A})$ can be defined as 
\begin{align*}
		\sigma_{k}(\mathds{A})=\frac{1}{k!}\delta^{i_{1}\ldots i_{k}}_{j_{1}\ldots j_{k}}\mathds{A}^{i_{1}}_{j_{1}} \dots \mathds{A}^{i_{k}}_{j_{k}}.
	\end{align*}
	
	If $\mathds{A}$ is diagonalizable, the previous expression is equivalent to
		\begin{align}
		\sigma_{k}(\mathds{A}) := \sum_{i_{1}<i_{2}<...<i_{k}}\lambda_{i_{1}}\ldots\lambda_{i_{k}}, \quad
		\text{for} \quad 1 \leq k \leq n, \label{defsigmakeigenvalues}
	\end{align}
	where $\lambda_1, \ldots, \lambda_n$ are the eigenvalues of a matrix $\mathds{A}$.
	
		We  define  $k$-th Newton transformation associated with $\mathds{A}$ as
	\begin{align*}
		%\label{definition_newton_tensor}
		T_{k}(\mathds{A}) = \sigma_{k}(\mathds{A})I - \sigma_{k-1}(\mathds{A})\mathds{A} + \ldots + (-1)^{k}\mathds{A}^{k}.
	\end{align*} 
	Equivalently, in terms of the components of $\mathds{A}$ we have
	\begin{align*}
		%\label{newton_tensor}
		T_{k}(\mathds{A})^{i}_{j} =\frac{1}{k!}\delta^{i_{1}...i_{k}i}_{j_{1}...j_{k}j}\mathds{A}^{i_{1}}_{j_{1}} \ldots \mathds{A}^{i_{k}}_{j_{k}},
	\end{align*}
	where $\delta^{i_{1}...i_{k}i}_{j_{1}...j_{k}j}$ is the generalized Kronecker delta symbol.\\
	
	It is not difficult to verify 
	\begin{align}
		\label{schouten_newton}
		T_{k-1}(\mathds{A})^{i}_{j} =\frac{\partial\sigma_{k}(\mathds{A})}{\partial \mathds{A}^{i}_{j}}.
	\end{align}
	The positive cone $\Gamma_{k}^{+}$ is defined by
	\begin{equation*}
	%	\label{cono_definition}
	\Gamma_{k}^{+}=\{ \mathds{A}\in \mathcal{M}_{n\times n}(\mathbb{R}): \, \sigma_j( \mathds{A})> 0\hbox{ where } j\in\{1, \ldots, k\}\}.
	\end{equation*}
	
The following properties of $\Gamma_{k}^{+}$ are well known  (e.g., see \cite{gaarding1959inequality}, \cite{Caffarelli1985TheDP}, \cite{5e65e644-6e2a-32bd-aeed-98a2562b5c2b}).
\begin{pro}
	\label{properties_cone}
	Each set $\Gamma_{k}^{+}$ is an open convex cone with vertex at the origin, and we have the following sequence of inclusions:
	\begin{equation*}
			\Gamma_{n}^{+} \subset \Gamma_{n-1}^{+} \subset \cdot \cdot \cdot \subset \Gamma_{1}^{+}.
	\end{equation*}
		
Furthermore, for symmetric linear transformations $\mathds{A}\in\Gamma_{k}^{+}$, $\mathds{B}\in\Gamma_{k}^{+}$, we have $t\mathds{A}+(1-t)\mathds{B}\in\Gamma_{k}^{+}$ for $t\in[0,1]$. If $\mathds{A}\in\Gamma_{k}^{+}$, then $T_{k-1}(\mathds{A})$ is positive definite and $\log(\sigma_{k}(\mathds{A}))$ and $\sigma_{k}(\mathds{A})^{1/k}$ are concave.	
	\end{pro}
For diagonalizable matrices, the positive cone $\Gamma_{k}^{+}$ can be alternatively characterized in terms of the eigenvalues $\lambda\in\mathbb{R}^{n}$ as follows (see for instance  \cite{wang2009k}).

\begin{pro}
	\label{cone_characterization}
	$\Gamma_{k}^{+}$ may also be equivalently defined as the component $\{\lambda\in\mathbb{R}^{n}\,|\,\sigma_{k}(\lambda)>0\}$ containing the vector $(1,\cdots,1)$, and characterized as
	\begin{equation*}
		%\label{cone_connected}
		\Gamma_{k}^{+}=\{\lambda\in\mathbb{R}^{n}\,|\,0<\sigma_{k}(\lambda)\leq\sigma_{k}(\lambda+\eta)\quad\text{for all } \eta\in\mathbb{R}^{n}, \,  \eta_{i}\leq0\}.
	\end{equation*}
\end{pro}

Recall the definition of the Schouten tensor given by \eqref{def Schouten}. Using the metric $g$, we may view the Schouten tensor $A_{g}$
	as an endomorphism of the tangent space at any point  and we consider the eigenvalues of the map $\mathds{A}=g^{-1}A_{g}$. Since the Ricci
	tensor is symmetric, these eigenvalues are real. With these eigenvalues we may use the definition of $\sigma_j$ given by Equation \eqref{defsigmakeigenvalues}
 and consequently, the associated positive cone can be defined as in  \eqref{cono_definition_g}. In that context, we will denote as $\sigma_j(g)$ the elementary symmetric polynomial evaluated at the eigenvalues of $g^{-1}A_{g}$ and the corresponding positive cone as $\Gamma_{k}^{+}$.
 \\

%The Schouten tensor satisfies the following transformation law under conformal changes.
Another relevant property in this work is how the Schouten tensor converts under conformal changes of metric. More precisely, 
let us denote $g_{v}=e^{-2 v}g$, then a direct computations reveals that the Schouten tensor of $g_{v}$ is related to the one of $g$ by the following transformation  law. 
	\begin{align}
		\label{transformation_law}
		A_{g_{v}}&=A_{g}+D^{2}v +dv\otimes dv-\frac{|Dv|_{g}^{2}}{2} g, 
	\end{align}
	where $D$ and $|\cdot|$ are computed with respect to the background metric $g$. \\
%\end{pro}

%\begin{pro}[\cite{viaclovsky1999conformal}]
%	\label{free_divergence}
%	Given any metric $g$, then the Newton tensor $T_{k}(g^{-1}A_{g})$ for $k\leq n-1$ is divergence-free with respect to the metric $g$, i.e,
%	\begin{align}
%		\label{divergencefree}
%		\sum_{i}\nabla_{i}(T_{k})_{j}^{i}=0 \quad \text{for all} \quad j.
%	\end{align}
%\end{pro}

We finish this section stating a previously known results that we use later in the paper.\\

\begin{pro}[Lemma 4.4 in  \cite{del2005singular}]\label{equationMar}
	Let  $M$ be conformally flat manifold with a metric	\\$g_{v}=v^{-2}|dx|^{2}$, $v>0$, then
	\begin{align*}
		%\label{equation_Mar}
		\ell\sigma_{\ell}(g_{v}) = v D_{j}\left( T_{\ell-1}^{ij} D_iv\right)-nT_{\ell-1}^{ij}D_i v D_j v+\frac{n-\ell+1}{2}\sigma_{\ell-1}(g_{v})|D v|^{2},
	\end{align*}
	where $D$ and $|\cdot|$ are computed with respect to the Euclidean background metric and $T_{\ell}$ is $\ell$-th Newton transformation associated with $g_{v}^{-1}A_{g_{v}}$.
\end{pro}

\section{PDE formulation of $k$-Yamabe solitons and proof of Theorem \ref{PDE_formulation}}
\label{sec:pdeformulation}
Our aim in this section is to prove Theorem \ref{PDE_formulation}.

\begin{proof}[Proof of Theorem \ref{PDE_formulation}]
	We assume that the metric $g_u$ is globally conformally equivalent to  the flat metric on $\mathbb{R}^{n}$, rotationally symmetric and it satisfies \eqref{soliton_f}.  We consider the following formalism for the conformal factor 
	\begin{equation}
		\label{euclidean_conformal}
			g:=g_{u}=u^{\frac{4k}{n+2k}}|dx|^{2},
	\end{equation}
	 where the background metric is Euclidean and its Schouten tensor vanishes. In this case Equation \eqref{transformation_law} for the Schouten tensor reduces to
	\begin{align}
		\label{transformation_law_2}
		A_{g_{u}}=-\frac{2k}{n+2k}u^{-1}D^{2}u +\frac{2k(n+4k)}{(n+2k)^{2}}u^{-2}D u\otimes Du -\frac{2k^{2}}{(n+2k)^{2}}u^{-2}|Du|^{2}g,
	\end{align}
	where $D$ and $|\cdot|$ are computed with respect to the Euclidean background metric.\\
	
\noindent In spherical coordinates the metric $g_u$ can be expressed as 
	\begin{equation*}
		%\label{g=gu}
		g_{u}=u(r)^{\frac{4k}{n+2k}}(dr^{2}+r^{2}g_{\mathbb{S}^{n-1}}),  \text{ where }  r=|x|.
	\end{equation*}
		
\noindent	 Let us consider the following change of coordinates
	\begin{align}\label{relwandu}
		w(s)=r^{2}u(r)^{\frac{4k}{n+2k}}, \quad \quad r=e^{s}.
	\end{align}
	Then the metric is equivalent to 
	\begin{equation}
		\label{g_cyl}
		g=w(s)g_{cyl},
	\end{equation}
	where $g_{cyl}=ds^{2}+g_{\mathbb{S}^{n-1}}$ is the cylindrical metric. 
	
	We  use a subindex $1$ or $s$ to refer to the $s$ direction and indices $2,3,...,n$ to refer to the spherical directions. By Equation \eqref{soliton_f} we have 
	\vspace{-0.1cm}
	\begin{align}
		\label{3.2}
		(\sigma_{k}^{1/k}(g)-\rho)g_{ij}=\nabla_{i}\nabla_{j}\varphi.
	\end{align}
	For a potential function $\varphi$ which is radially symmetric, or equivalently, that only depends on $s$ holds  
%It holds
%	\vspace{-0.1cm}
%	\begin{align*}
%		\nabla_{i}\nabla_{j}\varphi=\varphi_{ij}+\Gamma_{ij}^{l}\varphi_{l} \quad \text{and} \quad  \Gamma_{ij}^{l}=\frac{g^{ml}}{2}\left(\frac{\partial g_{im}}{\partial x_{j}}+\frac{\partial g_{jm}}{\partial x_{i}}-\frac{\partial g_{ij}}{\partial x_{m}} \right).
%	\end{align*}
%In this case $\varphi$ satisfies 
	\begin{align*}
		\nabla_{s}\nabla_{s}\varphi=\varphi_{ss}-\Gamma_{ss}^{s}\varphi_{s} \quad \text{and} \quad \nabla_{i}\nabla_{j}\varphi=-\Gamma_{ij}^{s}\varphi_{s}, \quad i\neq1 \text{ or } j\neq 1.
	\end{align*}
	Since
	\begin{align*}
		\Gamma_{ss}^{s}=\frac{w_{s}}{2w}, \quad \Gamma_{ii}^{s}=-\frac{w_{s}}{2w}, \quad i\neq1
	\end{align*}
	we conclude that
	\begin{align*}
		\nabla_{s}\nabla_{s}\varphi=\varphi_{ss}-\frac{w_{s}\varphi_{s}}{2w} \quad \text{and} \quad \nabla_{i}\nabla_{i}\varphi=\frac{w_{s}\varphi_{s}}{2w}, \quad i\neq1
	\end{align*}
	and the remaining derivatives vanish.
	Substituting the last two relations into \eqref{3.2} yields
	\begin{align}
		\label{3.3}
		\varphi_{ss}-\frac{w_{s}\varphi_{s}}{2w}=(\sigma_{k}^{1/k}(g)-\rho)w \quad \text{and} \quad \frac{w_{s}\varphi_{s}}{2w}=( \sigma_{k}^{1/k}(g)-\rho)w.
	\end{align}
	If we subtract the second equation from the first we get 
	\begin{align*}
		\varphi_{ss}-\frac{w_{s}\varphi_{s}}{w}=0.
	\end{align*}
	This is equivalent to $\left(\frac{\varphi_{s}}{w}\right)_{s}=0$  (since $w>0$) which implies 
	\begin{align}
		\label{potencial}
		\frac{\varphi_{s}}{w}=C.
	\end{align}
	The second expression in \eqref{3.3} and \eqref{potencial} imply that
	\begin{align*}
		%\label{3.5}
		w_{s}=\frac{2}{C}(\sigma_{k}^{1/k}(g)-\rho)w.
	\end{align*}			
	Setting $\theta=\frac{C}{2}$ we conclude that $w$ satisfies the equation		
	\begin{align*}
		%\label{sigma_{k}_equation}
		\sigma_{k}^{1/k}(g)-\theta\frac{w_{s}}{w}-\rho=0.
	\end{align*}		
	Recalling \eqref{g_cyl} we have
	\begin{align*}
		%\label{3.7}
		\sigma_{k}^{1/k}\left((g_{cyl})^{-1}A_{g} \right)=\theta w_{s}+w\rho.
	\end{align*}
Now we rewrite the above equation in terms of $u$. From \eqref{relwandu} we have that
	$$w_s=r(r^2 u^{\frac{4k}{n+2k}})_r,$$
	which implies 
	$$\theta w_s+ \rho w=(2\theta +\rho)r^2 u^{\frac{4k}{n+2k}} +  \frac{4k\theta}{n+2k} r^3u^{\frac{4k}{n+2k}-1} u_r.$$
	Since, for a radial function holds  $D u=u_r \frac{x}{r}$, we have
	\begin{align}\theta w_s+ \rho w &=r^2 u^{\frac{2k-n}{n+2k}}\left[(2\theta +\rho) u +  \frac{4k\theta}{n+2k}  x\cdot D u\right] \notag \\&= w u^{-1} \left[(2\theta +\rho) u +  \frac{4k\theta}{n+2k}  x\cdot D u\right]. \label{rhspdefor}\end{align}
	From \eqref{transformation_law_2} follows 
	\begin{multline}
		\sigma_{k}^{1/k}(g)= u^{-\frac{4k}{n+2k}}\sigma_k^{1/k}\left(-\frac{2k}{n+2k}u^{-1}D^{2}u \right.\\\left.+\frac{2k(n+4k)}{(n+2k)^{2}}u^{-2}Du\otimes Du-\frac{2k^{2}}{(n+2k)^{2}}u^{-2}|Du|^{2}I\right).
	\end{multline}
	Note that for every $a\in \mathbb{R}$ holds
	$$u^{-a}D^{2}(u^{a})=au^{-1}D^{2}u+a(a-1)u^{-2}Du\otimes Du.$$
	Taking $a=\frac{-2k}{n+2k}$ we have
	\begin{align*}
		\sigma_{k}^{1/k}(g)&= u^{-\frac{4k}{n+2k}}\sigma_k^{1/k}\left(u^{-a}D^{2}(u^{a})-\frac{u^{-2a}}{2}|D(u^{a})|^{2}I\right)\\
		&=u^{-\frac{4k}{n+2k}-a}
		\sigma_k^{1/k}\left(D^{2}(u^{a})-\frac{u^{-a}}{2}|D(u^{a})|^{2}I\right).
	\end{align*}
	Combining \eqref{rhspdefor} and the previous equation we obtain
	\begin{equation}
		\label{hessian_equation}
		\begin{split}
			\sigma_k^{1/k}\left(D^{2}(u^{a})-\frac{u^{-a}}{2}|D(u^{a})|^{2}I\right)
			=u^{\frac{2k}{n+2k}-1} \left[(2\theta +\rho) u +  \frac{4k\theta}{n+2k}  x\cdot D u\right],	
		\end{split}
	\end{equation}
	which implies \eqref{elliptic_equation}. 
	
	Now, we rewrite Equation \eqref{elliptic_equation} using Proposition \ref{equationMar}, when the conformal factor is given by $u^{\frac{4k}{n+2k}}$. It holds
	\begin{equation}
		\label{divergencia_equation}
		\begin{split}
			k\sigma_{k}(g)=u^{a}D_{j}\left(T^{ij}_{k-1}D_i (u^{a})\right)-nT^{ij}_{k-1} D_i (u^{a}) D_j (u^{a}) +\frac{n-k+1}{2}\sigma_{k-1}|D (u^{a})|^{2},
		\end{split}
	\end{equation}
	where $T_{k-1}$ is $(k-1)$-th Newton
	transformation associated with $g^{-1}A_{g}$.\\
	
	From \eqref{hessian_equation} and \eqref{divergencia_equation} we obtain a PDE formulation for our problem in quasi-divergence form given by
	\begin{align*} 
		&u^{a}D_{j}\left(T^{ij}_{k-1}D_i (u^{a})\right)-nT^{ij}_{k-1} D_i (u^{a}) D_j (u^{a})+\frac{n-k+1}{2}\sigma_{k-1}|D (u^{a})|^{2}\\ &=ku^{\frac{k(n-2k)}{n+2k}} \left[(2\theta +\rho) u +  \frac{4k\theta}{n+2k}  x\cdot D u\right]^k.
	\end{align*}
	Observe also, that if $g_{u}$ is a radially symmetric smooth solution of Equation \eqref{elliptic_equation}, then the above discussion (done backwards) implies that $g_{u}$ satisfies
	the $k$-Yamabe gradient soliton equation \eqref{soliton_f} with potential function $\varphi$ defined in terms of $w$ by \eqref{potencial}.
	We finish the proof of Theorem \ref{PDE_formulation} with the following claim.
	\begin{claim}
		\label{positividad_theta}
		If $g_{u}=u^{\frac{4k}{n+2k}}|dx|^{2}$ defines a complete $k$-Yamabe gradient soliton, then $2\theta+\rho>0$.
	\end{claim}
In fact, from  Equation \eqref{elliptic_equation} the condition $\sigma_{k}(g_{u})>0$ implies that right hand of this equation is positive for every $x$ in the manifold. In particular, when $x=0$ we obtain $(2\theta+\rho)u(0)>0$, concluding  $2\theta+\rho>0$ since by definition $u>0$.
This finishes  the proof of Theorem \ref{PDE_formulation}.
\end{proof}

\subsection{PDE formulation in terms of the eigenvalues of $A_{g_{u}}$ in the radial case}\text{ } 
From the previous section and taking into account that $u$ is a radial function it is not difficult to compute  that eigenvalues for the operator $$u^{-a}D^{2}(u^{a})-\frac{u^{-2a}}{2}|D(u^{a})|^{2}I$$ are 
$$au^{-1}u_{rr}+\frac{a^{2}}{2}(a-2)u^{-2}u_{r}^{2} \quad \text{and} \quad au^{-1}\frac{u_{r}}{r}-\frac{a^{2}}{2}u^{-2}u_{r}^{2}.$$
%where $\widetilde{\lambda}_{1}$ and $\widetilde{\lambda}_{2}$ are eigenvalues of operator $u^{-a}D^{2}(u^{a})$. Specifically

%\begin{align*}
%	\widetilde{\lambda}_{1}=au^{-1}u_{rr}+a(a-1)u^{-2}u_{r}^{2} \quad \text{and} \quad \widetilde{\lambda}_{2}=au^{-1}\frac{u_{r}}{r}.
%\end{align*}
Then, if we denote by $\lambda_1$ and $\lambda_2$ the eigenvalues of $g^{-1}A_{g}$,  they are given by
\begin{align}  \label{lambda1}
	\lambda_1&=-\left(\frac{1-m}{2}\right)\left(\frac{u_{rr}}{u}-\frac{(5-m)}{4}\frac{u_{r}^{2}}{u^{2}}\right)\hbox{ with multiplicity } 1
\end{align}
and
\begin{align} \label{lambda2}
	\lambda_2&=-\left(\frac{1-m}{2}\right)\frac{u_{r}}{u}\left(\frac{1}{r}+\frac{1-m}{4}\frac{u_{r}}{u}\right) \hbox{ with multiplicity } n-1.
\end{align}
Here $m=\frac{n-2k}{n+2k}$.\\
%Note that the eigenvalues of $g^{-1}A_{g}$ are $\lambda_1$ with multiplicity $1$ and $\lambda_{2}$ with multiplicity $(n-1)$.

In this context,  from \eqref{defsigmakeigenvalues} we have
\begin{equation*}
	\label{eigenvalues}
	\begin{split}
		\sigma_{k}(g)=\lambda_{2}^{\frac{k-1}{k}}&\left[\binom{n-1}{k-1}\lambda_1+\binom{n-1}{k}\lambda_2\right]^{1/k}.
	\end{split} 	
\end{equation*}
Then, Equation \eqref{hessian_equation} can be expressed as
\begin{equation*}
	\begin{split}
		\lambda_{2}^{\frac{k-1}{k}}\left[\binom{n-1}{k-1}\lambda_1+\binom{n-1}{k}\lambda_2\right]^{1/k}=(2\theta +\rho) u^{\frac{4k}{n+2k}} +  (1-m)\theta u^{\frac{4k}{n+2k}}\frac{ru_{r}}{u}		
	\end{split}
\end{equation*}
or
\begin{equation}
	\label{eigenvalues_equation}
	\begin{split}
		\lambda_{1}+\frac{n-k}{k}\lambda_{2}=\binom{n-1}{k-1}^{-1}\lambda_{2}^{1-k}u^{\frac{4k^{2}}{n+2k}}\left(2\theta+\rho+(1-m)\theta\frac{ru_{r}}{u}\right)^{k}.
	\end{split} 	
\end{equation}

\subsection{Admissibility of solutions}
In this subsection, we establish conditions for the admissibility of our solution. Taking into account the definition of the positive cone from \eqref{cono_definition_g} and Equation \eqref{eigenvalues_equation} expressed in terms of the eigenvalues, we show a sufficient condition to ensure that the obtained solution remains within the cone.
\begin{lem}
	\label{decrece_admisible}
	Let $g_{u}$ be given by  \eqref{euclidean_conformal} with $u$ a rotationally symmetric function  with $\sigma_{k}(g_{u})>0$ and $1\leq k\leq n$. 
	Then  $g_{u}$ belongs to the positive cone $\Gamma_{k}^{+}$ if and only if
	$\lambda_{2}>0$.
\end{lem}
\begin{proof}
	Assume first that $g_{u} \in \Gamma_{k}^{+}$. From 
	Proposition \ref{properties_cone} we know that $T_{l-1}(g_{u})$ is positive definite for every $l\in\{1,\ldots,k\}$. Since $u$ is radially symmetric we have  
	\begin{equation*}
		\sigma_l(\lambda_{1},\lambda_{2}):=\sigma_l(g_{u})=\lambda_{2}^{l-1}\left[\binom{n-1}{l-1}\lambda_1+\binom{n-1}{l}\lambda_2\right].\label{radialformsigma}
	\end{equation*}
	Therefore
	\begin{equation}
		\label{derivative_sigma}
		\frac{\partial\sigma_{l}(\lambda_{1},\lambda_{2})}{\partial\lambda_{1}}=\binom{n-1}{l-1}\lambda_2^{l-1}=T_{l-1}(\lambda_{1},\lambda_{2})_{1}^{1}>0, \quad l\in\{1,\dots, k\}.
	\end{equation}
	In particular, for $l=2$ this implies $\lambda_{2}>0$. Here we have used  \eqref{schouten_newton} and that tensor $g_{u}^{-1}A_{g_{u}}$ is diagonal with eigenvalues $\lambda_{1}$ and $\lambda_{2}$, with multiplicities $1$ and $n-1$,
	respectively.\\
	
	Assume now that $\lambda_2>0$. If $\lambda_1\geq 0$ we directly conclude that  $g_{u} \in \Gamma_{k}^{+}$ (since automatically $\sigma_l>0$ for every $l$), hence we may assume that $\lambda_1<0$. From \eqref{derivative_sigma} we have that   $\sigma_{k}$ is increasing in $\lambda_{1}$ and  $\sigma_{k}(\lambda,\lambda_{2},\dotso,\lambda_{2})>0$ for every $\lambda>\lambda_1$. In particular, 
	$(\lambda_2,\lambda_{2},\dotso,\lambda_{2})$ is in the same connected component of 
	$$\mathcal{C}=\{\lambda\in\mathbb{R}^{n}\,|\,\sigma_{k}(\lambda)>0\}$$  as $(\lambda_1,\lambda_{2},\dotso,\lambda_{2})$. 
	The characterization of the cone $\Gamma_{k}^{+}$ given in Proposition \ref{cone_characterization} concludes the proof of the Lemma.
\end{proof}

\section{ODE Analysis}
\label{sec:ODEanalysis}

In this section, we present the set up that will be used in the rest of the paper.  We follow the approach  in \cite{vazquez2006smoothing}, that was used to study the case $k=1$. More precisely,  in Subsection \ref{sec:derphaseanal} we define an autonomous system of ordinary differential equations and find its critical points in Subsection  \ref{critical_points}.
Solutions to our system will be understood as orbits of the phase plane, that join critical points  of the system or approach asymptotes, hence it is necessary to 
study the nature of the critical points  (Subsection \ref{subsec:linearization}). However, some of these points are not regular enough to perform the analysis in a standard way and to prove existence of the desired orbits we need to carefully set up a fixed point argument near the origin, which will be performed in Subsection \ref{subsec:linearization}.\\

% leading up to the proof of Theorems \ref{classification}, \ref{classification2} and \ref{classification3} that will be completed in the coming sections and it based on the phase plane analysis of an autonomous system of ordinary differential equations. We follow the approach of \cite{vazquez2006smoothing} for the semilinear case $k=1$. 
%We start this section by deriving the system that we need to analyze (Subsection \ref{sec:derphaseanal}) and establish the existence of critical points (Subsection \ref{critical_points}). To prove existence of solutions we need to study the nature of these critical points through linearization and derive a suitable choice of orbits in the different cases (Subsections \ref{orbit_choice} and \ref{subsec:linearization}). To prove existence near the origin, in Subsection \ref{subsec:linearization} we set up a fixed point argument, based on the computations in Subsection \ref{subsec:AnalysisCriticalPoints}.\\	

\noindent Throughout this section recall that    $$m=\frac{n-2k}{n+2k}<1 \hbox{ and }g_{u}=u(|x|)^{\frac{4k}{n+2k}}|dx|^{2},$$ 
where  $u$ is a solution of \eqref{elliptic_equation}.

\subsection{Derivation of the phase-plane system}\label{sec:derphaseanal} 
From Lemma \ref{decrece_admisible} and expression $\eqref{lambda2}$ we have 
$$
\lambda_2=-\left(\frac{1-m}{2}\right)\frac{u_{r}}{u}\left(\frac{1}{r}+\frac{1-m}{4}\frac{u_{r}}{u}\right)>0.
$$
Equivalently, $\lambda_2>0$ if and only if
\begin{equation}
	\label{range_X}
	-\frac{4}{1-m}\leq\frac{ru_{r}}{u}<0.	
\end{equation}
Then we can define the following positive functions:
\begin{align}
	\label{variables}
	&r=e^{s}, \quad 
	X(r)=\left(-\frac{ru_{r}}{u}\right)^{k}, \quad 
	Z(r)=(r^{2}u^{1-m})^{k}.
\end{align}
In terms of $X$,  Inequality \eqref{range_X}  is equivalent to  $$0<X^{1/k}\leq\frac{4}{1-m}.$$
From Equation \eqref{eigenvalues_equation} we have that $\sigma_{k}>0$ implies 
\begin{equation}
X^{1/k}<\left(\frac{n+2k}{k}\right)\left(\frac{2\theta+\rho}{4\theta}\right). \label{eq:restsigmapos}
\end{equation}
Since $0<X^{1/k}$ and  $2\theta+\rho>0$ (from Claim  \ref{positividad_theta}), in order to have an admissible solution Equation \eqref{eq:restsigmapos} imposes
 $$\theta>0.$$
It is also useful to define the parameters
\begin{align}
	\label{parameters}
	\gamma=\frac{2\theta+\rho}{\beta}=\left(\frac{n+2k}{k}\right)\left(\frac{2\theta+\rho}{4\theta}\right) \quad \text{and} \quad \beta=(1-m)\theta.
\end{align}

\noindent We define the admissible region  as
\begin{equation*}
	%\label{admissibility_region}
	\mathcal{A}=\left\{(x,z)\in \mathbb{R}^2: z>0 \quad \text{and} \quad  0<x<\min \{\gamma^k, X_A\}\right\},
\end{equation*}
where $\gamma$ is defined as in \eqref{parameters} and
 \begin{equation} X_A:=\left(\frac{n+2k}{k}\right)^{k}=\left(\frac{4}{1-m}\right)^k.\label{defxa}\end{equation}
We seek for solutions that satisfy $(X(s), Z(s))\in \mathcal{A}$ for every $s\in \mathbb{R}$.\\

A relevant remark is that the value of $\min \{\gamma^k, X_A\}$ only depends of the sign of $\rho$ and its relation with the parameter $\theta$, but not on $n-2k$. 
The admissible region will be depicted  in Subsection \ref{critical_points}  (see Figure \ref{picture_admissible}), along with the critical points of the system.\\
%We depict the admissible regions  (for different choices of $\rho$) in Figure \ref{picture_admissible} when $n>2k$, in which case $B$ is a critical point of system \eqref{dynamic_system}. For $n\leq 2k$ the picture only defers only in the position of  $B$ (that either does not exist or is outside of the admissible region).  

%\begin{remark}
%	\label{beta_positive}
%	 Claim \ref{positividad_theta} shows that $2\theta+\rho>0$. Then, Equation \eqref{eq:restsigmapos}
%implies that $\theta>0$.
%\end{remark}

With these definitions  we  deduce that $X$ and $Z$ satisfy the system
\begin{lem}
	\label{dynamic_system}
	\begin{equation}
		\label{system}
		\left\{
		\begin{gathered}
			X_{s}=-(n-2k)\left(1-\frac{k}{n+2k}X^{1/k}\right)X+Z\, f(X^{1/k})\hfill\\
			Z_{s}=2kZ\left(1-\frac{2k}{n+2k}X^{1/k}\right),\hfill
		\end{gathered}
		\right.
	\end{equation}
	where
	\begin{equation}
		\label{function_f}
		f(X^{1/k}):=c_{n,k}\beta^{k}\left(1-\frac{k}{n+2k}X^{1/k}\right)\left(\frac{\gamma-X^{1/k}}{1-\frac{k}{n+2k}X^{1/k}}\right)^{k},
	\end{equation}
	$\beta=(1-m)\theta$ and $c_{n,k}:=\frac{n+2k}{2^{k}\binom{n-1}{k-1}}\left(\frac{k}{n+2k}\right)^{1-k}$.	
\end{lem}
\begin{proof}
	Note that
	\begin{align}
				\frac{d r}{d s}&= r, \nonumber\\
			XZ &=(-r^{3}u^{-m}u_{r})^{k}, \nonumber\\
			(X^{1/k})_{s} &=-r^{2}\frac{u_{rr}}{u}+X^{1/k}+X^{2/k}, \label{derivateX}\\
			Z_s&=r Z_{r}=k Z^{(k-1)/k} \left[ 2r^2u^{1-m} - (1-m)(XZ)^{1/k}\right]. \label{Z_s}
	\end{align}
	Since $Z^{1/k}= r^2u^{1-m}$ we have that  \eqref{Z_s} is the second equation of \eqref{system}.\\
	
	\noindent To deduce the first equation of  \eqref{system}, we observe from \eqref{derivateX} that
	\begin{align*}
		-r^{2}\frac{u_{rr}}{u}=(X^{1/k})_{s}+X^{2/k}-X^{1/k}.
	\end{align*}
	If we multiply by $r^{2}$ the eigenvalues $\lambda_1$ and $\lambda_2$ (defined in \eqref{lambda1} and \eqref{lambda2})  we obtain the following.
	\begin{align*}
		r^2\, \lambda_1&=\frac{1-m}{2} \left(\frac{1}{k}X^{(1-k)/k}X_s+\frac{1-m}{4}X^{2/k}-X^{1/k}\right),\\
		r^2\, \lambda_2&= \frac{1-m}{2}X^{1/k} \left(1-\frac{1-m}{4}X^{1/k} \right).
	\end{align*}
	Multiplying the right side of equation \eqref{eigenvalues_equation} by $r^2$ as well we have
	\begin{align*}
		r^2\left[(2\theta +\rho) u^{\frac{4k}{n+2k}} +  (1-m)\,\theta\, u^{\frac{4k}{n+2k}}\frac{ru_{r}}{u} \right]=(2\theta +\rho)Z^{1/k} -  (1-m)\theta (ZX)^{1/k}.
	\end{align*}
	Finally, replacing $r^{2}\lambda_{1}$, $r^{2}\lambda_{2}$ and the previous equality in Equation \eqref{eigenvalues_equation}  we obtain the first line of System \eqref{system}.
	% For the second line of this it is sufficient to express $Z_{r}$ in terms of the variable $s$ from \eqref{Z_s}.
\end{proof}
\subsection{Critical points of System \eqref{system}}\text{ } 
\label{critical_points}
We are interested in analyzing the critical points of System \eqref{system} that are within the admissible region $\mathcal{A}$. Depending on the values of the parameters, we observe different behaviors that we consider as different cases.

\subsubsection*{ Case $n>2k$}

\medskip

In this situation there are at most three critical points. Indeed, the second line of System \eqref{system} selects the values $Z=0$ or 
\begin{equation} X^{1/k}_{B}=\frac{n+2k}{2k}. \label{defxb}\end{equation} For $Z=0$ we have from the first equation that $X=0$ or 
$X^{1/k}_{A}=\frac{n+2k}{k}$ (given as in \eqref{defxa}). Then, the critical points are   
\begin{align}
	\label{critical pts}
	0&=(0,0), \nonumber\\
	A&=\left(\left(\frac{n+2k}{k}\right)^{k},0\right). \nonumber
\end{align}
The other option is  point  \begin{equation}B=(X_{B},Z_{B}), \hbox{ where } Z_{B}=\frac{(n-2k)\binom{n-1}{k-1}}{k(2\rho)^{k}}.\label{pointB}\end{equation}
Note that this quantity can only be computed for $\rho\ne 0$, hence this is not a critical point (or is a ``critical point at infinity'') for $\rho=0$. \\

Using the relations given by \eqref{eq:restsigmapos},  for $\rho\leq 0$ we observe that $\gamma\leq X^{1/k}_{B}$, hence there are no critical points in the admisible region $\mathcal{A}$ and $(0,0)\in \partial \mathcal{A}$.\\

On the other hand, for  $\rho>0$ we have that $B\in \mathcal{A}$ and $(0,0)\in \partial \mathcal{A}$. However, whether $A\in \mathcal{A}$ depends on the relation between $\theta$ and $\rho$. More precisely, 
 for
$\frac{2\theta+\rho}{4\theta}\leq 1$ holds $A\not\in \mathcal{A}$ and if $\frac{2\theta+\rho}{4\theta}> 1$ we have $A\in \mathcal{A}$. \\

The admissible region can be depicted as follows.

\begin{figure}[h]
\centering
\tikzset{every picture/.style={line width=0.75pt}} %set default line width to 0.75pt        

\begin{tikzpicture}[x=0.75pt,y=0.75pt,yscale=-1,xscale=1]
	%uncomment if require: \path (0,280); %set diagram left start at 0, and has height of 280
	
	%Straight Lines [id:da7171574604602438] 
	\draw    (14,200.6) -- (198,200.6) ;
	\draw [shift={(200,200.6)}, rotate = 180] [color={rgb, 255:red, 0; green, 0; blue, 0 }  ][line width=0.75]    (10.93,-3.29) .. controls (6.95,-1.4) and (3.31,-0.3) .. (0,0) .. controls (3.31,0.3) and (6.95,1.4) .. (10.93,3.29)   ;
	%Straight Lines [id:da6597224508321475] 
	\draw    (30,217.6) -- (30,25.6) ;
	\draw [shift={(30,23.6)}, rotate = 90] [color={rgb, 255:red, 0; green, 0; blue, 0 }  ][line width=0.75]    (10.93,-3.29) .. controls (6.95,-1.4) and (3.31,-0.3) .. (0,0) .. controls (3.31,0.3) and (6.95,1.4) .. (10.93,3.29)   ;
	%Shape: Circle [id:dp4397339842044423] 
	\draw  [fill={rgb, 255:red, 0; green, 0; blue, 0 }  ,fill opacity=1 ] (98,128.3) .. controls (98,126.81) and (99.21,125.6) .. (100.7,125.6) .. controls (102.19,125.6) and (103.4,126.81) .. (103.4,128.3) .. controls (103.4,129.79) and (102.19,131) .. (100.7,131) .. controls (99.21,131) and (98,129.79) .. (98,128.3) -- cycle ;
	%Shape: Circle [id:dp7266695417711333] 
	\draw  [fill={rgb, 255:red, 0; green, 0; blue, 0 }  ,fill opacity=1 ] (150,200.3) .. controls (150,198.81) and (151.21,197.6) .. (152.7,197.6) .. controls (154.19,197.6) and (155.4,198.81) .. (155.4,200.3) .. controls (155.4,201.79) and (154.19,203) .. (152.7,203) .. controls (151.21,203) and (150,201.79) .. (150,200.3) -- cycle ;
	%Straight Lines [id:da48990978421780507] 
	\draw [color={rgb, 255:red, 208; green, 2; blue, 27 }  ,draw opacity=1 ][line width=1.5]    (85,42.6) -- (85,217) ;
	%Straight Lines [id:da5860665860660663] 
	\draw [color={rgb, 255:red, 189; green, 16; blue, 224 }  ,draw opacity=1 ][fill={rgb, 255:red, 189; green, 16; blue, 224 }  ,fill opacity=1 ] [dash pattern={on 3.75pt off 3pt on 7.5pt off 1.5pt}]  (80,48.6) -- (33,84) ;
	%Straight Lines [id:da15374669410232267] 
	\draw [color={rgb, 255:red, 189; green, 16; blue, 224 }  ,draw opacity=1 ] [dash pattern={on 3.75pt off 3pt on 7.5pt off 1.5pt}]  (81,74.2) -- (34,111.6) ;
	%Straight Lines [id:da40831930325336274] 
	\draw [color={rgb, 255:red, 189; green, 16; blue, 224 }  ,draw opacity=1 ] [dash pattern={on 3.75pt off 3pt on 7.5pt off 1.5pt}]  (81,102.6) -- (34,139) ;
	%Straight Lines [id:da5243944034680079] 
	\draw [color={rgb, 255:red, 189; green, 16; blue, 224 }  ,draw opacity=1 ] [dash pattern={on 3.75pt off 3pt on 7.5pt off 1.5pt}]  (79,131.6) -- (32,169) ;
	%Straight Lines [id:da266115627644381] 
	\draw [color={rgb, 255:red, 189; green, 16; blue, 224 }  ,draw opacity=1 ] [dash pattern={on 3.75pt off 3pt on 7.5pt off 1.5pt}]  (80,157.6) -- (33,195) ;
	%Straight Lines [id:da18203701067849343] 
	\draw    (239,214.6) -- (239,22.6) ;
	\draw [shift={(239,20.6)}, rotate = 90] [color={rgb, 255:red, 0; green, 0; blue, 0 }  ][line width=0.75]    (10.93,-3.29) .. controls (6.95,-1.4) and (3.31,-0.3) .. (0,0) .. controls (3.31,0.3) and (6.95,1.4) .. (10.93,3.29)   ;
	%Straight Lines [id:da7701408051771548] 
	\draw    (231,200.6) -- (407,200.6) ;
	\draw [shift={(409,200.6)}, rotate = 180] [color={rgb, 255:red, 0; green, 0; blue, 0 }  ][line width=0.75]    (10.93,-3.29) .. controls (6.95,-1.4) and (3.31,-0.3) .. (0,0) .. controls (3.31,0.3) and (6.95,1.4) .. (10.93,3.29)   ;
	%Straight Lines [id:da4132107078698364] 
	\draw [color={rgb, 255:red, 208; green, 2; blue, 27 }  ,draw opacity=1 ][line width=1.5]    (321,36.6) -- (321,211) ;
	%Shape: Circle [id:dp37203745240225206] 
	\draw  [fill={rgb, 255:red, 0; green, 0; blue, 0 }  ,fill opacity=1 ] (289,129.3) .. controls (289,127.81) and (290.21,126.6) .. (291.7,126.6) .. controls (293.19,126.6) and (294.4,127.81) .. (294.4,129.3) .. controls (294.4,130.79) and (293.19,132) .. (291.7,132) .. controls (290.21,132) and (289,130.79) .. (289,129.3) -- cycle ;
	%Shape: Circle [id:dp6509484354550132] 
	\draw  [fill={rgb, 255:red, 0; green, 0; blue, 0 }  ,fill opacity=1 ] (365,200.3) .. controls (365,198.81) and (366.21,197.6) .. (367.7,197.6) .. controls (369.19,197.6) and (370.4,198.81) .. (370.4,200.3) .. controls (370.4,201.79) and (369.19,203) .. (367.7,203) .. controls (366.21,203) and (365,201.79) .. (365,200.3) -- cycle ;
	%Straight Lines [id:da46484165392069166] 
	\draw [color={rgb, 255:red, 189; green, 16; blue, 224 }  ,draw opacity=1 ][fill={rgb, 255:red, 189; green, 16; blue, 224 }  ,fill opacity=1 ] [dash pattern={on 3.75pt off 3pt on 7.5pt off 1.5pt}]  (313,45.6) -- (249,73.6) ;
	%Straight Lines [id:da2959650770953295] 
	\draw [color={rgb, 255:red, 189; green, 16; blue, 224 }  ,draw opacity=1 ][fill={rgb, 255:red, 189; green, 16; blue, 224 }  ,fill opacity=1 ] [dash pattern={on 3.75pt off 3pt on 7.5pt off 1.5pt}]  (314,71.6) -- (250,99.6) ;
	%Straight Lines [id:da17020682864154324] 
	\draw [color={rgb, 255:red, 189; green, 16; blue, 224 }  ,draw opacity=1 ][fill={rgb, 255:red, 189; green, 16; blue, 224 }  ,fill opacity=1 ] [dash pattern={on 3.75pt off 3pt on 7.5pt off 1.5pt}]  (313,98.6) -- (249,126.6) ;
	%Straight Lines [id:da22357113785637184] 
	\draw [color={rgb, 255:red, 189; green, 16; blue, 224 }  ,draw opacity=1 ][fill={rgb, 255:red, 189; green, 16; blue, 224 }  ,fill opacity=1 ] [dash pattern={on 3.75pt off 3pt on 7.5pt off 1.5pt}]  (313,128.6) -- (249,156.6) ;
	%Straight Lines [id:da08064456212943161] 
	\draw [color={rgb, 255:red, 189; green, 16; blue, 224 }  ,draw opacity=1 ][fill={rgb, 255:red, 189; green, 16; blue, 224 }  ,fill opacity=1 ] [dash pattern={on 3.75pt off 3pt on 7.5pt off 1.5pt}]  (311,152.6) -- (247,180.6) ;
	%Straight Lines [id:da7099975749112277] 
	\draw [color={rgb, 255:red, 189; green, 16; blue, 224 }  ,draw opacity=1 ][fill={rgb, 255:red, 189; green, 16; blue, 224 }  ,fill opacity=1 ] [dash pattern={on 3.75pt off 3pt on 7.5pt off 1.5pt}]  (314,171.6) -- (264,194.6) ;
	%Straight Lines [id:da2975673174558069] 
	\draw [color={rgb, 255:red, 189; green, 16; blue, 224 }  ,draw opacity=1 ][fill={rgb, 255:red, 189; green, 16; blue, 224 }  ,fill opacity=1 ] [dash pattern={on 3.75pt off 3pt on 7.5pt off 1.5pt}]  (314,189.6) -- (306.66,193.02) -- (299,196.6) ;
	%Straight Lines [id:da8743892835122866] 
	\draw [color={rgb, 255:red, 189; green, 16; blue, 224 }  ,draw opacity=1 ] [dash pattern={on 3.75pt off 3pt on 7.5pt off 1.5pt}]  (81,179.6) -- (58,198) ;
	%Straight Lines [id:da8082447490282301] 
	\draw    (444,212.6) -- (444,20.6) ;
	\draw [shift={(444,18.6)}, rotate = 90] [color={rgb, 255:red, 0; green, 0; blue, 0 }  ][line width=0.75]    (10.93,-3.29) .. controls (6.95,-1.4) and (3.31,-0.3) .. (0,0) .. controls (3.31,0.3) and (6.95,1.4) .. (10.93,3.29)   ;
	%Straight Lines [id:da8978384597940046] 
	\draw    (434,199.6) -- (605,199.6) ;
	\draw [shift={(607,199.6)}, rotate = 180] [color={rgb, 255:red, 0; green, 0; blue, 0 }  ][line width=0.75]    (10.93,-3.29) .. controls (6.95,-1.4) and (3.31,-0.3) .. (0,0) .. controls (3.31,0.3) and (6.95,1.4) .. (10.93,3.29)   ;
	%Shape: Circle [id:dp06156164612858883] 
	\draw  [fill={rgb, 255:red, 0; green, 0; blue, 0 }  ,fill opacity=1 ] (500,130.3) .. controls (500,128.81) and (501.21,127.6) .. (502.7,127.6) .. controls (504.19,127.6) and (505.4,128.81) .. (505.4,130.3) .. controls (505.4,131.79) and (504.19,133) .. (502.7,133) .. controls (501.21,133) and (500,131.79) .. (500,130.3) -- cycle ;
	%Shape: Circle [id:dp5400403145008374] 
	\draw  [fill={rgb, 255:red, 0; green, 0; blue, 0 }  ,fill opacity=1 ] (542,200.3) .. controls (542,198.81) and (543.21,197.6) .. (544.7,197.6) .. controls (546.19,197.6) and (547.4,198.81) .. (547.4,200.3) .. controls (547.4,201.79) and (546.19,203) .. (544.7,203) .. controls (543.21,203) and (542,201.79) .. (542,200.3) -- cycle ;
	%Straight Lines [id:da24684643405876217] 
	\draw [color={rgb, 255:red, 208; green, 2; blue, 27 }  ,draw opacity=1 ][line width=1.5]    (569,31.6) -- (569,155.6) -- (569,206) ;
	%Straight Lines [id:da6698926054667838] 
	\draw [color={rgb, 255:red, 189; green, 16; blue, 224 }  ,draw opacity=1 ][fill={rgb, 255:red, 189; green, 16; blue, 224 }  ,fill opacity=1 ] [dash pattern={on 3.75pt off 3pt on 7.5pt off 1.5pt}]  (536,40.6) -- (448,63.6) ;
	%Straight Lines [id:da3016557806609881] 
	\draw [color={rgb, 255:red, 189; green, 16; blue, 224 }  ,draw opacity=1 ][fill={rgb, 255:red, 189; green, 16; blue, 224 }  ,fill opacity=1 ] [dash pattern={on 3.75pt off 3pt on 7.5pt off 1.5pt}]  (538,72.6) -- (450,96.6) ;
	%Straight Lines [id:da9726241825198014] 
	\draw [color={rgb, 255:red, 189; green, 16; blue, 224 }  ,draw opacity=1 ][fill={rgb, 255:red, 189; green, 16; blue, 224 }  ,fill opacity=1 ] [dash pattern={on 3.75pt off 3pt on 7.5pt off 1.5pt}]  (542,103.6) -- (450,130.6) ;
	%Straight Lines [id:da9255102812130451] 
	\draw [color={rgb, 255:red, 189; green, 16; blue, 224 }  ,draw opacity=1 ][fill={rgb, 255:red, 189; green, 16; blue, 224 }  ,fill opacity=1 ] [dash pattern={on 3.75pt off 3pt on 7.5pt off 1.5pt}]  (543,136.6) -- (448,165.6) ;
	%Straight Lines [id:da05088769767003187] 
	\draw [color={rgb, 255:red, 189; green, 16; blue, 224 }  ,draw opacity=1 ][fill={rgb, 255:red, 189; green, 16; blue, 224 }  ,fill opacity=1 ] [dash pattern={on 3.75pt off 3pt on 7.5pt off 1.5pt}]  (542,168.6) -- (456,192.6) ;
	%Straight Lines [id:da46191802866269427] 
	\draw [color={rgb, 255:red, 189; green, 16; blue, 224 }  ,draw opacity=1 ][fill={rgb, 255:red, 189; green, 16; blue, 224 }  ,fill opacity=1 ] [dash pattern={on 3.75pt off 3pt on 7.5pt off 1.5pt}]  (543,186.6) -- (515,194.6) ;
	
	% Text Node
	\draw (96,108) node [anchor=north west][inner sep=0.75pt]   [align=left] {$B$};
	% Text Node
	\draw (150.7,206) node [anchor=north west][inner sep=0.75pt]   [align=left] {$X_{A}$};
	% Text Node
	\draw (76,221) node [anchor=north west][inner sep=0.75pt]  [color={rgb, 255:red, 208; green, 2; blue, 27 }  ,opacity=1 ] [align=left] {$\gamma ^{k}$};
	% Text Node
	\draw (196,204.6) node [anchor=north west][inner sep=0.75pt]   [align=left] {$X$};
	% Text Node
	\draw (16,6) node [anchor=north west][inner sep=0.75pt]   [align=left] {$Z$};
	% Text Node
	\draw (355.7,206) node [anchor=north west][inner sep=0.75pt]   [align=left] {$X_{A}$};
	% Text Node
	\draw (406,203.6) node [anchor=north west][inner sep=0.75pt]   [align=left] {$X$};
	% Text Node
	\draw (312,218) node [anchor=north west][inner sep=0.75pt]  [color={rgb, 255:red, 208; green, 2; blue, 27 }  ,opacity=1 ] [align=left] {$\gamma ^{k}$};
	% Text Node
	\draw (226,6) node [anchor=north west][inner sep=0.75pt]   [align=left] {$Z$};
	% Text Node
	\draw (531,204.3) node [anchor=north west][inner sep=0.75pt]   [align=left] {$X_{A}$};
	% Text Node
	\draw (601,203.6) node [anchor=north west][inner sep=0.75pt]   [align=left] {$X$};
	% Text Node
	\draw (429,5) node [anchor=north west][inner sep=0.75pt]   [align=left] {$Z$};
	% Text Node
	\draw (563,210) node [anchor=north west][inner sep=0.75pt]  [color={rgb, 255:red, 208; green, 2; blue, 27 }  ,opacity=1 ] [align=left] {$\gamma ^{k}$};
	% Text Node
	\draw (43,248) node [anchor=north west][inner sep=0.75pt]   [align=left] {Case $\rho < 0$};
	% Text Node
	\draw (253,249.6) node [anchor=north west][inner sep=0.75pt]   [align=left] {Case $0< \rho < 2\theta $};
	% Text Node
	\draw (478,250.6) node [anchor=north west][inner sep=0.75pt]   [align=left] {Case $\rho  >2\theta $};
	% Text Node
	\draw (286,109) node [anchor=north west][inner sep=0.75pt]   [align=left] {$B$};
	% Text Node
	\draw (497,110) node [anchor=north west][inner sep=0.75pt]   [align=left] {$B$};

\end{tikzpicture}
\caption{Admissible regions in terms of $\rho$ and $\theta$ when $n>2k$.}
\label{picture_admissible}
\end{figure}
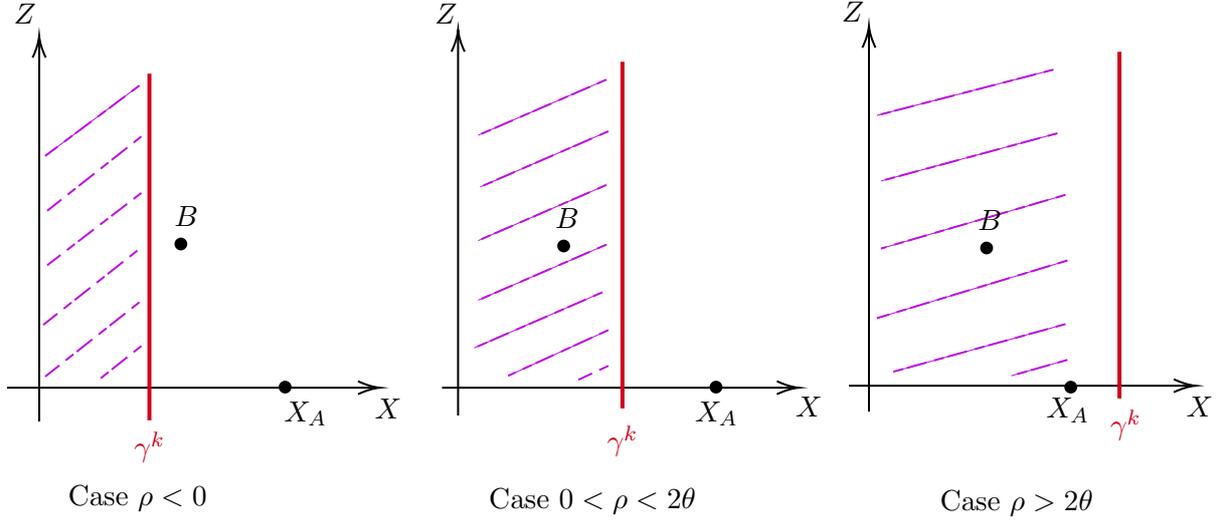
 
% \newpage

%It will also be relevant in our analysis to observe that
%$X_s$ vanishes at the isocline of vertical slopes 
%\begin{equation}
%	\begin{split}
%		\label{isocline}
%		Z&=\left(\frac{n-2k}{c_{n,k}\beta^{k}}\right)\left(\frac{1-\frac{k}{n+2k}X^{1/k}}{\gamma-X^{1/k}}%\right)^{k}X.\\
%	\end{split}
%\end{equation}
%This isocline is undefined when $X^{1/k}=\gamma$  (where $\gamma$ is given by \eqref{parameters})  and it vanishes for $X=0$ or $X=X_{A}$.

 We refer to  the line $X^{1/k}=\gamma$ as the 
 {\it asymptote} of the system, which is justified by the following proposition.
 
 \begin{pro}\label{pro:aysmptoten>2k} Assume that $n> 2k$ and $\rho\leq 2\theta$. Let $(X,Z)$ be a solution to System \eqref{system} such that there is an $s_0$ for which $X(s_0)<\gamma^k$ and $Z(s_0)>0$,
 then  $X(s)\leq\gamma^k$ and   $Z(s)> 0$  for every $s\geq s_0$. 
 \end{pro} 
 \begin{proof}

 We proceed by contradiction assuming that there is an  $s_1\geq s_0$ such that $ X(s_1)=\gamma^k$ 
 and $X(s)>\gamma^k$ for $s\in(s_1, s_1+\delta)$ (for some $\delta>0$). In particular, this implies $X_s(s_1)\geq 0$. On the other hand, from the first equation in \eqref{system} we have that
 $$X_s(s_1)= -(n-2k)\left(1-\frac{k}{n+2k}\gamma\right)\gamma^k<0,$$
 which is a contradiction.\\
 
 Now we use the second equation of \eqref{system} to show that
 $$(\ln Z)_s\geq 2k\left(1-\frac{(1-m)}{2}\gamma\right).$$
 Since $ 1-\frac{(1-m)}{2}\gamma=-\frac{\rho}{2\theta}$ we have
  \begin{equation}Z(s)\geq Z(s_0) e^{-\frac{k\rho}{\theta}(s-s_0)}>0 \hbox { for } s\geq s_0.\label{lowerboundz}\end{equation}
 
 \end{proof}
%{\textcolor{magenta}{la potencia me da $-\frac{k\rho}{\theta}(s-s_{0})$}}.
 \begin{remark}
Inequality  \eqref{lowerboundz} holds for every value of $\rho$, $n$ and $k$.
 \end{remark}

\subsubsection*{Case $n=2k$}

In this situation, the second line of System \eqref{system} selects the values $Z=0$ or $X^{1/k}_{B}=\frac{n+2k}{2k}$. Since $n=2k$, we observe that $(X,0)$ is a critical point of \eqref{system} for every value of $X$. On the other hand, for $X=X_B$ it is not possible to select a value $Z\ne 0$ such that the first line of 
 \eqref{system}  is 0. Then, for every $\rho$ the only possible critical points in $\overline{\mathcal{A}}$ are of the form $(X,0)$ with $X\in[0, \min\{\gamma^k, X_A\}]$.
 %For $n<2k$ {\textcolor{red}{and $\rho>0$ or even $k$}} we have $Z_B<0$ and $B\not \in \mathcal{A}$. 
 Note that the admissible region remains as depicted in Figure \ref{picture_admissible}, but the point $B$ is on the $X$-axis ($Z_B=0$).\\

 We also observe that in this case we have a particular solution given by $X=\gamma^k$ and $Z= Z(0) e^{k(2-\gamma) s}=e^{-\frac{k \rho}{\theta}s}.$  Uniqueness of solutions to \eqref{system} allow us to directly extend Proposition \ref{pro:aysmptoten>2k} to this situation.
 
  \begin{pro}\label{pro:aysmptoten=2k} Assume that $n=2k$ and $\rho\leq 2\theta$. Let $(X,Z)$ be a solution to System \eqref{system} such that there is an $s_0$ for which $X(s_0)<\gamma^k$ and $Z(s_0)>0$,
 then  $X(s)\leq\gamma^k$ and $Z(s)>0$  for every $s\geq s_0$. 
 \end{pro} 

  \begin{remark}
  The solution $X=\gamma^k$ and $Z= Z(0) e^{k(2-\gamma) s}= Z(0) e^{-\frac{k \rho}{\theta}s}$ can be translated into the solution  $u_\alpha (x)= \alpha |x|^{-\frac{2\theta+ \rho}{\theta}} $ to  \eqref{elliptic_equation} that is not defined at the origin.
  \end{remark}
 %{\textcolor{magenta}{el exponente de $\alpha$... $u(\alpha)=\alpha^{1/k}|x|^{-\frac{2\theta+\rho}{\theta}}$}}
  
 \subsubsection*{ Case $n<2k$}
\medskip
Proceeding as in the previous cases, it is not difficult to check that $O=(0,0)$ and $A=\left(\left(\frac{n+2k}{k}\right)^{k},0\right)$ are the only possible critical points with $Z\geq 0$. Moreover, these are the only points in $\overline{\mathcal{A}}$.
 
 \medskip

\subsection{Choice of orbits}\text{ } 
\label{orbit_choice}
Since we are looking for smooth solutions to  \eqref{system}, from the definition of $X$ and $Z$ given by \eqref{variables} we need to impose $\lim_{s\to -\infty} X(s) =\lim_{s\to -\infty} Z(s)=0$ (since this corresponds with the value at $r=0$). Hence we would like to analyze the existence of orbits that emerge from $(0,0)$.  We will show that these orbits may  be of three types:\\

\begin{itemize}
	\item {\bf Orbits of type $\gamma$}: These are orbits that approach $(0,0)$  as   $s\to -\infty$ and they reach the asymptote $X=\gamma^{k}$ as $s\to \infty$, that is  $$\lim_{s\to \infty} X(s) =\gamma^k\hbox{ and}
	\lim_{s\to \infty} Z(s)=\infty.$$
	
	\item   {\bf Orbits of type $B$}:  These are orbits that approach $(0,0)$  as   $s\to -\infty$ and $(X_B, Z_B)$ as $s\to \infty$. With an abuse of language, we will also refer to orbits that asymptotically stay in a suitably small neighborhood of $B$ as  (generalized) orbits of type $B$ as well, even if we cannot ensure convergence to this point (this statement will be clarified in Proposition \ref{pro:rhopositive}).
	
	\item   {\bf Orbits of type $A$}:  These are orbits that approach $(0,0)$  as   $s\to -\infty$ and $(X_A, 0)$ as $s\to \infty$. With an abuse of language, we will also refer to orbits that converge to points of the form $(X,0)$ as generalized orbits of type $A$ (which may occur when $n=2k$).\\
\end{itemize} 

Standard results of dynamical systems suggest that complete orbits should join two critical points or a critical point with the asymptote.
Considering the discussion in the previous section, when $\rho\leq 0$ and $n\geq 2k$ we expect that the desired solution  corresponds with an orbit of type $\gamma$, while for $\rho>0$ we may have different behaviors depending on the relations among the parameters: for $n>2k$ and $\rho\leq 2\theta$ orbits are of type B, while for $\rho> 2\theta$ can be either of type $B$ or type $A$. If $n=2k$ and $\rho>0$ orbits are necessarily  generalized orbits of type $A$. We will analyze this in detail in Subsection \ref{orbitsn=2k}. For $n<2k$ there is no asymptote and we expect orbits to join $(0,0)$ and $A$, but this would violate the admissibility condition if $\rho\leq 2\theta$.

\subsection{Analysis of critical points of system \eqref{system}}
\label{subsec:AnalysisCriticalPoints}	
The linearization of \eqref{system} around a critical point $(X,Z)$ has Jacobian matrix
\begin{align}
	\label{jacobian_matrix}
	J=
	\begin{pmatrix}
		\frac{\partial F}{\partial X} & \frac{\partial F}{\partial Z}\\
		\\
		\frac{\partial G}{\partial X} & \frac{\partial G}{\partial Z}
	\end{pmatrix}
\end{align}
where $F(X,Z)$ and $G(X,Z)$ are defined as the right hand side of System \eqref{system}, that is 
\begin{equation}
	\label{definition_F_G}
	\begin{split}
		F(X,Z)&:=-(n-2k)\left(1-\frac{k}{n+2k}X^{1/k}\right)X+Z\, f(X^{1/k}),\\
		G(X,Z)&:=2kZ\left(1-\frac{(1-m)}{2}X^{1/k}\right).
	\end{split}
\end{equation}
A direct computations implies 
\begin{align*}
	\frac{\partial F}{\partial X}& = (2k-n)+m(k+1)X^{1/k}+Z\frac{\partial f}{\partial X}, \\
	\frac{\partial F}{\partial Z} 
	&=f(X^{1/k}),\\
	\frac{\partial G}{\partial X}
	&=-(1-m)ZX^{\frac{1-k}{k}},\\
	\frac{\partial G}{\partial Z}
	&=2k-(1-m)kX^{1/k},
\end{align*}
where $$\frac{\partial f}{\partial X}= c_{n,k}\beta^{k} X^{\frac{1-k}{k}}\left(\frac{\gamma- X^{1/k}}{1-\frac{k}{n+2k}X^{1/k}}\right)^{k-1}\left\{\left(\frac{k-1}{n+2k}\right)\left(\frac{\gamma- X^{1/k}}{1-\frac{k}{n+2k}X^{1/k}}\right)-1
\right\}.$$
%\medskip

We observe that the functions $F$ and $G$ are not differentiable at $X=0$ nor at $X=X_A$, hence it is not possible to do a standard  linearization argument. We introduce a modification in the computations that will allow us to construct  
an approximate solution in a neighborhood of the origin and within the admissible region.
We define a restricted Jacobian as follows.

%We use a fixed point method to prove existence of solutions with the desired behavior. {\color{red} We consider the following definition

\begin{defi}
Let $R\subseteq \mathbb{R}^2$ open and $P\in \bar{R}$, then the restricted Jacobian at $P$, denoted by $\left. J\right|_R(P)$ is the matrix that satisfies for every $\epsilon>0$ there exists a $\delta>0$ such that
\begin{equation*}
	%\label{restrictedjacobian}
	\hbox{ If }\| (X,Z)- P\|< \delta \hbox{ and } (X,Z)\in R \hbox{ then }\|J(X,Z)- \left. J\right|_R(P)\|<\epsilon.
\end{equation*}
Here the norm of $2\times 2$ matrices  is taken by considering this space as equivalent to  $\mathbb{R}^4$.

\end{defi}

\medskip

\subsubsection*{\bf Critical point $(0,0)$}\text{ } 

\medskip

As previously discussed the Jacobian is not well defined at the origin, but we  restrict to a region 
 where  the quotient $\frac{Z}{X}$ is bounded. More precisely, we have the following.
 
 % it is possible to compute the limit as we approach to the origin. Moreover, recall that we are only interested in studying the region where $X>0$.  In this context we have the following.

%\begin{align}
%	\label{cota}
%	\lim_{(X,Z)\to(0,0)}ZX^{(1-k)/k}=\lim_{(X,Z)\to(0,0)}\frac{Z}{X}X^{1/k}=0.
%\end{align}

\begin{pro}
	\label{linealization_cero}
Fix $K>0$ and define $R_K=\{ ( X,  Z)\in \mathcal{A}\, :\,  Z<  KX \}$. Then
	\begin{equation*}
	\left. J\right|_{R_K}(O)=\begin{pmatrix}
	-(n-2k) & f(0)\\
	0 & 2k
	\end{pmatrix}.		
	\end{equation*}
We observe that  $O$ is a saddle point (unstable) as long as $n>2k$, it degenerates when $n=2k$ and it is an unstable node (source) when $n<2k$. Moreover, the eigenvalues in this case are $\tau_{1}=2k$ and $\tau_{2}=-(n-2k)$ with their respective eigenvectors $$\vec{e_{1}}=\left(1,\frac{n}{f(0)}\right) \quad \text{and} \quad \vec{e_{2}}=(1,0).$$	
\end{pro}

\subsubsection*{\bf Critical point $B=(X_{B},Z_{B})$}\text{ } 

\medskip

In this case, by a direct computation we obtain
\begin{pro}
	\label{linearization_B}
	The linearization of \eqref{system} around $B$ (given by \eqref{defxb}  and \eqref{pointB})  has Jacobian matrix 
	$$
	J(B)=\begin{pmatrix}
		-\frac{n-2k}{2}\left[\left(\frac{n+2k}{2k}\right)^{k}-\frac{k-1}{k}\right] & f(X_B^{\frac{1}{k}}) \\
		-\frac{n-2k}{f(X_B^{\frac{1}{k}})}  & 0
	\end{pmatrix},
	$$
	where the determinant and the trace of the previous matrix are equal to
	$$D=n-2k \quad \text{and} \quad T=	-\frac{n-2k}{2}\left[\left(\frac{n+2k}{2k}\right)^{k}-\frac{k-1}{k}\right]$$
	respectively. 
	\end{pro}

Recall from Section \ref{critical_points} that the critical point $B$ only appears if $n>2k$. In that context we have that the determinant $D$ is strictly positive.
	Since $\left( \frac{n+2k}{2k}\right)^{k}>1>\frac{k-1}{k}$, it is not difficult to verify that the trace $T$ is strictly negative. This implies that the real part of the eigenvalues of $J(B)$ is strictly negative, hence $B$  is an attractor.

\subsubsection*{\bf Critical point $A=(X_{A}, 0)$}\text{ } 

\medskip

As in our analysis for $O$, we need to restrict the region to study  the critical point $A$. We also consider the following change of variables: Let  $$W:=\left(\frac{n+2k}{k}-X^{1/k}\right)^{k}\hbox{ and }V:=Z.$$ With this convention, it is clear that $W$ is positive to the left of $X_{A}$ and  $W=0$ when $X^{1/k}=\frac{n+2k}{k}$. Then the point $A$ corresponds with  $W=V=0$.\\

%Then the first line of   \eqref{system}  reduces to 

%$$W_{s}=(n-2k)W(1-W^{1/k})-  Z h(W^{\frac{1}{k}}),
%$$
%where $$h(W^{1/k})=c_{n,k} \beta^k (1-W^{1/k}) 
%\left(\frac{\nu-W^{\frac{1}{k}}}{1-W^{1/k}}\right)^{k},$$
%$\nu=\frac{k\gamma}{n+2k}=\frac{2\theta +\rho}{4 \theta}$ and $c_{n,k}$ is the constant in the definition of the function $f$ in \eqref{function_f}.

System \eqref{system} is rewritten as follows
\begin{align}
	\label{system_A}
	\left\{
	\begin{gathered}
		W_{s}=(n-2k)W\left(1-\frac{k}{n+2k}W^{1/k}\right)-  V h(W^{\frac{1}{k}}),\hfill\\
		V_{s}=-2kV\left(1-\frac{2k}{n+2k}W^{1/k}\right),\hfill
	\end{gathered}
	\right.
\end{align}
where \begin{equation}\label{defh}h(W^{1/k})=c_{n,k} \beta^k \left(1-\frac{k}{n+2k}W^{1/k}\right) 
\left(\frac{\nu+W^{\frac{1}{k}}}{1-\frac{k}{n+2k}W^{1/k}}\right)^{k}.\end{equation}
Here $c_{n,k}$ is the constant in  \eqref{function_f} and $\nu=\gamma-X_A^{\frac{1}{k}}$.

The linearization of \eqref{system_A} around a critical point in the variables $(W,V)$ has Jacobian matrix given by \eqref{jacobian_matrix}
replacing the variable $X$ by $W$, $Z$ with $V$ and, with abuse of notation, denoting $F$ and $G$ the right hand-side of first and second line in \eqref{system_A}, respectively. We obtain
\begin{align*}
	\frac{\partial F}{\partial W}& = (n-2k)-m(k+1)W^{1/k}-V\frac{\partial h}{\partial W}, \\
	\frac{\partial F}{\partial V} 
	&=-h(W^{1/k}),\\
	\frac{\partial G}{\partial W}
	&=(1-m)VW^{\frac{1-k}{k}},\\
	\frac{\partial G}{\partial V}
	&=-2k+(1-m)kW^{1/k},
\end{align*}
where $$\frac{\partial h}{\partial W}= c_{n,k}\beta^{k} W^{\frac{1-k}{k}}\left(\frac{\nu+ X^{1/k}}{1-\frac{k}{n+2k}W^{1/k}}\right)^{k-1}\left\{\left(\frac{k-1}{n+2k}\right)\left(\frac{\nu+ W^{1/k}}{1-\frac{k}{n+2k}W^{1/k}}\right)+1
\right\}.$$

%\begin{align*}\frac{\partial F}{\partial W}& = (n-2k)\left[1-\frac{k+1}{k}W^{\frac{1}{k}}\right]-Z\frac{\partial h}{\partial W},\\  
%	\frac{\partial F}{\partial Z} 
%	&=-h(W^{\frac{1}{k}}),
%\end{align*}
%where 
%\begin{align*}
%	\frac{\partial h}{\partial W}&= \left[-c_{n,k} \beta^k +
%k\left(\frac{\nu-W^{\frac{1}{k}}}{1-W^{1/k}}\right)^{k-1}\frac{\nu-1}{1-W^{1/k}}\right]\frac{W^{1/k-1}}{k}.
%\end{align*}	
%In addition
%\begin{align*}
%	\frac{\partial G}{\partial W}=2ZW^{(1-k)/k} \quad \text{and} \quad 
%	\frac{\partial G}{\partial Z}=-2k(1-W^{1/k}).
%\end{align*}
Similar to the analysis at the critical point (0,0), we restrict to a region where
 the quotient $\frac{V}{W}$ is bounded. Abusing again the notation, we refer to the admissible region transformed into the variables  $W$ and $V$ as $\mathcal{A}$ as well.
  By a direct computation we obtain the following result.
\begin{pro}
	 Fix $K>0$ and define $R_K=\{ ( W,  V)\in \mathcal{A}\, :\,  V< KW \}$. Then the restricted Jacobian of  \eqref{system_A} around $\left(0,0\right)$  is given by 
	\begin{equation*}
		\left.J\right|_{R_K}(0,0)=	\begin{pmatrix}
			n-2k & -h(0)\\
			0 & -2k
		\end{pmatrix}.
	\end{equation*}
 This implies that $A$ is a saddle point as long as $n>2k$, it degenerates for $n=2k$ and is a source node when $n<2k$. The eigenvalues in this case are $\tau_{1}=-2k$ and $\tau_{2}=n-2k$ with the respective eigenvectors $$\vec{e_{1}}=\left(1, \frac{h(0)}{n}\right) \quad \text{and} \quad \vec{e_{2}}=(1,0).$$
\end{pro}

\subsection{The fixed point argument for existence at the origin}\text{ } 
\label{subsec:linearization}
The singular behavior of the dynamical system \eqref{system} at the origin requires a delicate  analysis to ensure that  the desired orbits emanates from the origin.
We finish this section by showing the existence of such orbits 
 and this will be obtained from a fixed point argument.\\

The main result of this section reads as follows.
 
 \begin{thm}\label{existence near the origin}
For every $\alpha>0$ there is an $s_0$ such that a  solution $(X,Z)$ to \eqref{system} 
   exists for $s\in(-\infty, s_0]$ and  satisfies
 \begin{align*}
 \lim_{s\to-\infty}X(s)&=\lim_{s\to-\infty}Z(s)=0,\\
  \lim_{s\to-\infty}e^{-2ks}X(s)&=\frac{n}{f(0)}  \lim_{s\to-\infty}e^{-2ks}Z(s)=\alpha_k,\\
	Z(s)>0,  & \quad X(s)>0,
\end{align*}
where $\alpha_k=\alpha^{(1-m)k}$.

\end{thm}

\begin{remark}
	\label{initial_condition_alpha}
The choice of $\alpha_k$ is related to the definition of $Z$ in \eqref{variables}. More precisely, it is chosen such that if
$$\lim_{s\to -\infty} e^{-2ks}Z(s)=\lim_{r\to 0} u^{(1-m) k}(r)=\alpha_k,$$
then $u(0)=\alpha$.
\end{remark}

From linearization around critical point $(0,0)$ in Subsection \ref{subsec:AnalysisCriticalPoints} the system \eqref{system}
can be rewritten as:
\begin{align*}
%	\label{system_new}
	\left\{
	\begin{gathered}
		X_{s}=(2k-n)X+Zf(0)+F_{1}(s)\hfill\\
		Z_{s}=2kZ+G_{1}(s),\hfill
	\end{gathered}
	\right.
\end{align*}
where
\begin{align}
	\label{h1}
	F_{1}(X,Z)&=kmX^{(k+1)/k}+Z(f(X^{1/k})-f(0)),\\ 
	\label{h2}
	G_{1}(X,Z)&=-k(1-m)ZX^{1/k},
\end{align} 
and $f(X^{1/k})$ is defined by \eqref{function_f}. 

Let $Q$ be the matrix that represents linearization around $(0,0)$:
$$Q:=\begin{pmatrix}
	-(n-2k) & f(0)\\
	0 & 2k
\end{pmatrix}.$$
Then, System \eqref{system} can be written as 
\begin{equation}
	\label{exponential_equation}
	\left(e^{-Qs}\begin{pmatrix}
		X \\
		Z
	\end{pmatrix} \right)_{s}=e^{-Qs}\begin{pmatrix}
		F_{1}(X,Z) \\
		G_{1}(X,Z)
	\end{pmatrix},
\end{equation}
where $e^{-Qs}=Ve^{-Ds}V^{-1}$ with $V$ the matrix which columns are eigenvectors of matrix $Q$, $V^{-1}$ is the inverse matrix of $V$ and $D$ is a diagonal matrix with eigenvalues of $Q$ in the diagonal. That is,
\begin{equation*}
	e^{-Qs}=\begin{pmatrix}
		e^{-(2k-n)s} & -\frac{f(0)}{n}e^{-(2k-n)s}+\frac{f(0)}{n}e^{-2ks}\\
		0 & e^{-2ks}
	\end{pmatrix}.
\end{equation*}

By integrating Equation \eqref{exponential_equation} between $-\infty$ and $s$, we obtain:
\begin{equation}
	\label{general_solution}
	\begin{pmatrix}
		X \\
		Z
	\end{pmatrix}=c_{1}e^{2ks}\begin{pmatrix}
		1\\
		\frac{n}{f(0)}
	\end{pmatrix}+c_{2}e^{(2k-n)s}\begin{pmatrix}
		1 \\
		0
	\end{pmatrix}+\text{lower order terms}.
\end{equation}
 We are interested in solutions of the system that follow the direction of the eigenvector $(1,\frac{n}{f(0)})$, i.e. those with rapid decay determined by the eigenvalue $2k$, with $c_{2}=0$ in \eqref{general_solution}. Consistently, we define the space and operator to which the fixed-point argument will be applied.\\

Let $C^{0}((-\infty,s_{0}];\mathbb{R}^{2})$ be the space of continuous functions with domain $(-\infty,s_{0}]$  and range contained in $\mathbb{R}^2$. Fix $\alpha>0$ as let $\alpha_k$ as in Theorem \ref{existence near the origin}. We define
\begin{align*}
%	\label{space_fixed_point}
		\mathcal{B}_{s_{0}, \alpha}:=\left\{\begin{pmatrix}
			X \\
			Z
		\end{pmatrix}\in C^{0}((-\infty,s_{0}];\mathbb{R}^{2}): 	
		\right. & \frac{\alpha_k }{2} \leq e^{-2ks}\,X\leq 2 \,\alpha_k ,  \notag \\ 	
		& \, \left. \qquad \frac{n\alpha_k}{2 f(0)}\leq e^{-2ks}\,Z\leq 2\frac{n\alpha_k}{f(0)} \right\}.
\end{align*}
We consider this space with norm of the supremum 
$$\left\|\begin{pmatrix}
	X \\
	Z
\end{pmatrix}\right\|_{\infty}=\sup\limits_{-\infty< t\leq s_{0}}\left\{|X(t)|, |Z(t)|\right\}.$$

Note that integrating Equation \eqref{exponential_equation} with respect to $s$ and considering the solution of the system that follows the direction of the eigenvector $(1,\frac{n}{f(0)})$ shows that \eqref{exponential_equation} is equivalent to the following integral equation 
\begin{equation*}
%	\label{integral_equation}
	e^{-Qs}\begin{pmatrix}
		X(s) \\
		Z(s)
	\end{pmatrix} = \alpha_k\begin{pmatrix}
		1 \\
		\frac{n}{f(0)}
	\end{pmatrix}+ \int_{-\infty}^{s}e^{-Qt}\begin{pmatrix}
		F_{1}(X(t),Z(t)) \\
		G_{1}(X(t),Z(t))
	\end{pmatrix}dt.
\end{equation*}

Hence, we define the operator

\begin{equation}
	\label{operador}
	\quad E\begin{pmatrix}
		X(\cdot) \\
		Z(\cdot)
	\end{pmatrix}(s):=\alpha_k e^{2ks}\begin{pmatrix}
		1 \\
		\frac{n}{f(0)}
	\end{pmatrix}+ \int_{-\infty}^{s}e^{Q(s-t)}\begin{pmatrix}
		F_{1}(X(t),Z(t)) \\
		G_{1}(X(t),Z(t))
	\end{pmatrix}dt.
\end{equation}
We first observe that for $(X,Z)\in \mathcal{B}_{s_{0}, \alpha}$ we have that $F_{1}(X(t),Z(t)) $,	$G_{1}(X(t),Z(t))$ decay like $e^{(2k+1)t}$ as $t\to -\infty$ and 
the integrand in \eqref{operador} decays as  $e^{t}$, hence it is integrable and the functional is well defined (this bound will be computed more carefully in Lemma \ref{E_into_itself}).

We show using the Banach fixed point Theorem that for every $\alpha$ and a suitable choice of $s_0$ the functional $E$ has fixed point in $\mathcal{B}_{s_{0}, \alpha}$. We need several intermediate results.

\begin{lem}
	\label{Banach_space}
	The space $(\mathcal{B}_{s_{0}, \alpha},\|\cdot\|_{\infty})$ is a complete metric space.	
\end{lem}
\begin{proof}
Recall from \eqref{variables} that $r=e^s$. Then, we can identify  $\mathcal{B}_{s_{0}, \alpha}$ with the space 
$$\mathcal{E}_{s_{0},\alpha}=\left\{(X, Z) \in C^{0}([0, \, e^{s_0}];\mathbb{R}^{2})\,:\, \frac{\alpha_k}{2} \leq r^{-2k}X\leq 2\alpha_k,  \, \frac{n\alpha_k}{2 f(0)}\leq r^{-2k}Z\leq 2 \frac{n\alpha_k}{ f(0)}\right\}.$$
Note that the decay condition as  $s\to -\infty$ implies that  $\lim_{s\to -\infty} X(s)=\lim_{s\to -\infty} Z(s)=0$, or equivalently, for $(X,Z)\in \mathcal{E}_{s_{0},\alpha}$, $X(0)=Z(0)=0$. The  decay bounds also impose a modulus of continuity for $X(r)$ and $Z(r)$ at $r=0$.
It is easy to prove that this is a closed subspace of the 
continuous function space from $[0, \, e^{s_0}]$ to $\mathbb{R}^2$. Since $\mathcal{E}_{s_{0},\alpha}$ is complete the desired result easily follows.

% $(\mathcal{B}_{s_{0}, \alpha},\|\cdot\|_{\infty})$ is {\textcolor{violet}{complete}}.
\end{proof}

The following lemma is a direct consequence of  the differentiability of the function $f$ away from $X_A$.

\begin{lem}
	\label{f_bounded}
	Let $f$ be the function defined by \eqref{function_f}. %with  $Y:=X^{1/k}$. 
	Assume that  $0\leq Y\leq\min\{\gamma, X_{B}\}$, then there is an $M$ (that is uniform in the parameters) such that $$|f(Y)-f(0)|\leq M\,Y.$$
	
In addition, for  $(X,Z)\in \mathcal{B}_{s_0, \alpha}$ holds $X(s)\leq 2e^{2ks_0}\alpha_k.$
Then for $$s_0\leq s_1:=  (2k)^{-1}\ln(\alpha^{-1}_k \min\{\gamma, X_{B}\}/2)$$ we have that
the condition $0\leq X^{\frac{1}{k}}(s)\leq\min\{\gamma, X_{B}^{\frac{1}{k}}\} 	$ is satisfied when $s\leq s_0$.
\end{lem}

Now we prove that $E$ satisfies the conditions of the Banach fixed point Theorem. We denote throughout this section $$E=(E_1, E_2).$$
We first prove that the operator $E$ maps  $ \mathcal{B}_{s_0, \alpha}$ into itself.
\begin{lem}
	\label{E_into_itself} There is an $s_2\in \mathbb{R}$ such that if $s_0\leq s_2$ and
	 $\begin{pmatrix}
		X \\
		Z
	\end{pmatrix}\in\mathcal{B}_{s_{0}, \alpha}$ then $E\begin{pmatrix}
		X \\
		Z
	\end{pmatrix}\in\mathcal{B}_{s_{0}, \alpha}$. 	
\end{lem}

\begin{proof}
	Note that from \eqref{operador}
	\begin{equation*}
		\begin{split}
			\left\|E\begin{pmatrix}
				X(s) \\
				Z(s)
			\end{pmatrix}\right\|_{\infty}\leq\alpha_k e^{2ks}\begin{pmatrix}
				1 \\
				\frac{n}{f(0)}
			\end{pmatrix} + 	\int_{-\infty}^{s}\left\|e^{Q(s-t)}\begin{pmatrix}
				F_{1}(\cdot,\cdot) \\
				G_{1}(\cdot,\cdot)
			\end{pmatrix} 
			\right\|dt,
		\end{split}
	\end{equation*}
	where
	\begin{equation*}
		e^{Q(s-t)}\begin{pmatrix}
			F_{1}(\cdot,\cdot) \\
			G_{1}(\cdot,\cdot)
		\end{pmatrix} =\begin{pmatrix}
			e^{(2k-n)(s-t)}\left(F_{1}-\frac{f(0)}{n}G_{1}\right)+\frac{f(0)}{n}e^{2k(s-t)}G_{1} \\
			e^{2k(s-t)}G_{1}.
		\end{pmatrix}.
	\end{equation*}
	Replacing \eqref{h1} and \eqref{h2}  in the previous equation we have
	\begin{equation*}
		\begin{split}
			E_{2}(X(s),Z(s))%&=\frac{\alpha_k n}{f(0)}e^{2ks}+\int_{-\infty}^{s}e^{(2k-n)(s-t)}G_{1}dt\\
			&=\frac{\alpha_k n}{f(0)}e^{2ks}-k(1-m)\int_{-\infty}^{s}e^{(2k-n)(s-t)}ZX^{1/k}dt.
		\end{split}
	\end{equation*}
	Hence
	\begin{equation}  \label{boundE_21}
		\begin{split}
			E_{2}&(X(s),Z(s))\leq \frac{\alpha_k n}{f(0)}e^{2ks}.
		\end{split}
	\end{equation}

	Since $(X,Z)\in\mathcal{B}_{s_{0}, \alpha}$ we have that $ZX^{1/k}(t)\leq \frac{(2\alpha_k)^{1+\frac{1}{k}n}}{f(0)} e^{2(k+1)t}$, which implies
	
	\begin{equation}  \label{boundE_22}
		\begin{split}
			E_{2}&(X(s),Z(s))\geq \frac{\alpha_k n}{f(0)}e^{2ks}\left(1-\frac{k(1-m)2^{1+1/k}\alpha_k^{1/k}}{(n-2k+2)n}e^{2s}\right).
		\end{split}
	\end{equation}
	Similarly, 
	\begin{equation*}
		%\label{cota_E1}
		\begin{split}
			E_{1}(X(s),Z(s))=\alpha_k e^{2ks} +\int_{-\infty}^{s}\left(e^{(2k-n)(s-t)}\left(F_{1}-\frac{f(0)}{n}G_{1}\right)+\frac{f(0)}{n}e^{2k(s-t)}G_{1}\right)dt.
		\end{split}
	\end{equation*}		
	Using Lemma \ref{f_bounded}  and that $(X,Z)\in\mathcal{B}_{s_{0}, \alpha}$ we can show that
	
	\begin{equation} \label{boundE_11}
		\begin{split}
			E_{1}(X(s),Z(s))&\leq \alpha_k e^{2ks}+\frac{M_{1}}{n+2}e^{2(k+1)s}
		\end{split}
	\end{equation}
	and
	\begin{equation}  \label{boundE_12}
		\begin{split}
			&E_{1}(X(s),Z(s))\geq \alpha_k e^{2ks}-\left(\frac{(2\alpha_k)^{(k+1)/k}k(1-m)}{2}-\frac{m_{1}}{n+2}\right)e^{2(k+1)s},
		\end{split}
	\end{equation}
where $M_{1}:=\max\left\{k(2\alpha_k)^{(k+1)/k},\frac{(2\alpha_k)^{1+1/k}Mn}{f(0)}\right\}$ and $m_{1}:=k\left(\frac{\alpha_k}{2}\right)^{(k+1)/k}-\frac{\alpha_k n}{2f(0)}(2\alpha_k)^{1/k}M$.\\
	
	Choosing $s_2$ suitably close to $-\infty$ (and depending only on the parameters $n$, $k$, $\theta$, $\rho$ and $\alpha$), we have that for $s_0\leq s_2$ holds
	\begin{align*}\frac{\alpha_k}{2} & \leq e^{-2ks}E_{1}(X(s),Z(s))\leq 2\alpha_k\hbox{ and  }\\
	\frac{\alpha_k n}{2 f(0)}&\leq e^{-2ks}E_{2}(X(s),Z(s))\leq 2\frac{\alpha_k n}{f(0)}\end{align*} for $s\leq s_0$. Equivalently,  $$ E(X(s),Z(s))\in \mathcal{B}_{s_0, \alpha}.$$
	\end{proof}

	\begin{remark}\label{asymptotic behavior of E}
	From \eqref{boundE_21}, \eqref{boundE_22}, \eqref{boundE_11} and \eqref{boundE_12} we have  for every $(X(s),Z(s))\in \mathcal{B}_{s_0, \alpha}$ that
	$$\lim_{s\to -\infty} e^{-2ks}E\begin{pmatrix}
		X(\cdot) \\
		Z(\cdot)
	\end{pmatrix}(s)= \alpha_k \begin{pmatrix}
		1 \\
		\frac{n}{f(0)}
	\end{pmatrix}.$$
	
	\end{remark}

\begin{lem}
	\label{contraction}
	There is an $s_{3}\in \mathbb{R}$  such that  for every $s_0\leq s_3$ the operator $$E:\mathcal{B}_{s_{0}, \alpha}\to\mathcal{B}_{s_{0}, \alpha}$$ is  contractive.	
\end{lem}
\begin{proof}
	We need to estimate
	\begin{equation}
		\label{inequality_contraction}
		\begin{split}
			&	\left\|E\begin{pmatrix}
				X(s) \\
				Z(s)
			\end{pmatrix}-E\begin{pmatrix}
				V(s) \\
				W(s)
			\end{pmatrix}\right\| \\ &\leq	\int_{-\infty}^{s}e^{Q(s-t)}\left\|\begin{pmatrix}
				F_{1}(X(t),Z(t)) \\
				G_{1}(X(t),Z(t))
			\end{pmatrix}-\begin{pmatrix}
				F_{1}(V(t),W(t)) \\
				G_{1}(V(t),W(t))
			\end{pmatrix}
			\right\|dt,
		\end{split}
	\end{equation}
	for $(X,Z)$ and $(V,W)$ in $\mathcal{B}_{s_{0}, \alpha}$.
	
	%\sout{Note that, the first coordinate of the integrand in the previous inequality is given by} {\color{green} ELIMINAR TODA LAS ECS DE ABAJO
	%\begin{equation*}
	%	\begin{split}
	%		&e^{(2k-n)(s-t)}[km(X^{1+1/k}-V^{1+1/k})+Z(f(X^{1/k})-f(0))+W(f(0)-f(V^{1/k}))]\\
	%		&+e^{(2k-n)(s-t)}\frac{f(0)k(m-1)}{n}(WV^{1/k}-ZX^{1/k}) +e^{2k(s-t)}\frac{f(0)k(m-1)}{n}(ZX^{1/k}-WV^{1/k}).
	%	\end{split}
	%\end{equation*}
%\sout{	Similarly, the value of the integrand's second coordinate is}
%	\begin{equation*}
%		\begin{split}
%			-k(1-m)e^{2k(s-t)}(ZX^{1/k}-WV^{1/k}).
%		\end{split}
%	\end{equation*}		}
	
%From the definition of $F_1$ and $G_1$ in \eqref{h1}-\eqref{h2}, we first  need to estimate the quotient $\frac{X^{1/k}-V^{1/k}}{X-V}$  for $(X, Z), \,(V, W)\in \mathcal{B}_{s_{0}, \alpha}$.

The following estimate will be useful in the coming computations.
\begin{claim}
	\label{quotient}
	 For $(X, Z), \,(V, W)\in \mathcal{B}_{s_{0}, \alpha}$ we have $$|X^{1/k}-V^{1/k}|\leq e^{-2(k-1)s}|X-V|.$$
\end{claim}
\begin{proof}
Since $$X-V=(X^{1/k}-V^{1/k})(X^{\frac{k-1}{k}}+X^{\frac{k-2}{k}}V^{1/k}+\cdots +X^{1/k}V^{\frac{k-2}{k}}+V^{\frac{k-1}{k}}),$$
	the hypothesis $X>\frac{1}{2}\alpha_k e^{2ks}$ and $V>\frac{1}{2}\alpha_k e^{2ks}$ implies
	\begin{equation}
		\label{cota_binomio}
		(X^{\frac{k-1}{k}}+X^{\frac{k-2}{k}}V^{1/k}+\cdots +X^{1/k}V^{\frac{k-2}{k}}+V^{\frac{k-1}{k}})\geq \left(\frac{\alpha_k}{2}\right)^{\frac{k-1}{k}}e^{2(k-1)s},
	\end{equation}
from which the inequality follows.
\end{proof}

%{\color{green} ELIMINAR TODO ESTO
%\sout{Using this inequality and that $|W|\leq2\frac{n\alpha}{f(0)}e^{2ks}$ we have the following estimate for the second coordinate}
%	\begin{equation*}
%		\begin{split}
%			&|k(m-1)e^{2k(s-t)}(ZX^{1/k}-WV^{1/k})|\\
%			&=|k(m-1)|e^{2k(s-t)}|W(X^{1/k}-V^{1/k})+X^{1/k}(Z-W)|\\&\leq |k(m-1)|e^{2k(s-t)}\left(\frac{4n}{f(0)}\left(\frac{c_{1}}{2}\right)^{1/k}e^{2t}|X-V|+(2\alpha)^{1/k}e^{2t}|Z-W|\right)\\
%			&\leq C_{1}e^{2ks+(2-2k)t}(|X-V|+|Z-W|),
%		\end{split}
%	\end{equation*}	
%\sout{	where $C_{1}=\left|k(m-1)2^{3-1/k}\alpha^{1/k}\frac{n}{f(0)}\right|$.}}

From the definition of $G_1$ \eqref{h2}, Claim \ref{quotient} and $|Z|, \,|W|\leq2\frac{n\alpha_{k}}{f(0)}e^{2ks}$ we have
$$|G_1(X, Z)-G_1(V, W)|\leq C_{1}e^{2ks+(2-2k)t}(|X-V|+|Z-W|),$$ where $C_1$ (that can be computed explicitly) only depends on the parameters ($n$, $k$, $\rho$ and $\alpha$).

Similarly, from  the definition of $F_1$ \eqref{h1}, Claim \ref{quotient}, Equation \eqref{cota_binomio} and Lemma \ref{f_bounded} we have
$$ |F_1(X, Z)-F_1(V, W)|\leq  [C_{2}e^{(2k-n)(s-t)}+C_{3}e^{2ks+2(1-k)t}](|X-V|+|Z-W|),$$
where $C_2$ and $C_3$ are constants that can be explicitly computed in terms of the parameters ($n$, $k$, $\rho$ and $\alpha$).

	By integrating the two previous expressions and substituting into \eqref{inequality_contraction}, we obtain\\
	\begin{equation*}
		\begin{split}
			\left\|E\begin{pmatrix}
				X \\
				Z
			\end{pmatrix}-E\begin{pmatrix}
				V \\
				W
			\end{pmatrix}\right\|_{\infty}  &\leq	\sup\{C_{1},C_{2},C_{3}\}e^{2s}\left\|\begin{pmatrix}
				X \\
				Z
			\end{pmatrix}-\begin{pmatrix}
				V \\
				W
			\end{pmatrix}\right\|_{\infty}.
		\end{split}
	\end{equation*}
Let $C:=\sup\{C_{1},C_{2},C_{3}\}$ and pick $s_3$ such that  $Ce^{2s_3}<1$. For such a value of $s_3$ we conclude the result.
\end{proof}
\begin{thm} \label{thm:fixedpoint}
	Let $(\mathcal{B}_{s_{0}, \alpha},\|\cdot\|_{\infty})$ with $s_{0}:=\min\{s_{1},s_{2},s_{3}\}$. Then $E:\mathcal{B}_{s_{0}, \alpha}\to\mathcal{B}_{s_{0}, \alpha}$ has a
	unique fixed point $(\overline{X},\overline{Z})\in\mathcal{B}_{s_{0}, \alpha}$ such that $$\left\|E^{n}\begin{pmatrix}
		X \\
		Z
	\end{pmatrix}-\begin{pmatrix}
		\overline{X} \\
		\overline{Z}
	\end{pmatrix}\right\|\leq\frac{c^{n}}{1-c}\left\|E\begin{pmatrix}
		X \\
		Z
	\end{pmatrix}-\begin{pmatrix}
		X \\
		Z
	\end{pmatrix}\right\|, \quad \begin{pmatrix}
		X \\
		Z
	\end{pmatrix}\in\mathcal{B}_{s_{0}, \alpha},$$
where $c$ is the contraction constant in Lemma \ref{contraction}.
\end{thm}

\begin{proof}

	The proof is a direct consequence of Lemmas \ref{Banach_space}, \ref{E_into_itself}, \ref{contraction} and the contraction principle. Moreover, $(\overline{X},\overline{Z})\in\mathcal{B}_{s_{0}, \alpha}$ implies that 
		
	$$\lim\limits_{s\to-\infty}\begin{pmatrix}
		\overline{X}(s) \\
		\overline{Z}(s)
	\end{pmatrix}=\begin{pmatrix}
		0 \\
		0
	\end{pmatrix}.$$

We conclude that $(X,Z)$ is a solution to initial value problem \eqref{system} with initial condition at minus infinity given by $(0,0)$, or equivalently, there is an orbit emanating from the origin.
\end{proof}

From Theorem \ref{thm:fixedpoint} and Remark \ref{asymptotic behavior of E} we can prove for every $n$ that

	$$\lim_{s\to -\infty} e^{-2ks}E^n\begin{pmatrix}
		X(\cdot) \\
		Z(\cdot)
	\end{pmatrix}(s)= \alpha_k \begin{pmatrix}
		1 \\
		\frac{n}{f(0)}
	\end{pmatrix}.$$

In particular, we have following property of the fixed point of the operator $E$.
\begin{lem}
	\label{condition_minus_infty}
	The fixed point $(\overline{X},\overline{Z})$ satisfies $$\lim\limits_{s\to-\infty}e^{-2ks}\begin{pmatrix}
		\overline{X} \\
		\overline{Z}
	\end{pmatrix}=\alpha_k\begin{pmatrix}
		1 \\
		\frac{n}{f(0)}
	\end{pmatrix}.$$
\end{lem}

Theorem \ref{existence near the origin} follows directly from Theorem \ref{thm:fixedpoint} and Lemma \ref{condition_minus_infty}.

\section{Global existence of orbits}
\label{sec:globalexistence}

In this section we show the existence of the desired orbits and thus proving  Theorem \ref{classification} and  Theorem \ref{classification3}. % when $n>2k$ in the cases  $\rho<0$, $\rho>0$ and $\rho=0$.
 From Theorem \ref{existence near the origin} we have the existence of orbits emerging from the origin. Moreover,   noting that
 the functions $F$ and $G$ (given by \eqref{definition_F_G}) are smooth for points $(X_0, Z_0)\in \mathcal{A}$ (and away from $O$ and $A$), standard ODE theory implies 
 that the orbits can be uniquely continued from $(X(s_0), Z(s_0))$ (where $s_0$ is the value in Theorem  \ref{existence near the origin}) as long as $(X(s), Z(s))\in \mathcal{A}$. 
 What remains to study is when solutions remain admissible for every $s\in \mathbb{R}$ and, if this is the case, 
  their asymptotic behavior as $s\to \infty$. We devote the rest of the section to this task.
 \\

\subsection{Existence of orbits in the case $n>2k$}\label{solitonsn>2k}

\subsubsection{$k$-Yamabe expander soliton $\rho<0$.} \label{sec:expanders}\text{ } 
As we explained in Subsection \ref{orbit_choice}, the desired orbits for the case of $\rho<0$  should depart from the origin and reach the asymptote given by $X^{1/k}=\gamma$  (the admissible orbits must be of type $\gamma$). In this subsection, we  show their existence.\\

Recall from Section \ref{critical_points} that for $\rho<0$ we have that
 $\gamma=\left(\frac{2\theta+\rho}{2\theta}\right)X_{B}^{1/k}\leq X_{B}^{1/k}$ and there are no critical points in the admissible region. From Proposition \ref{pro:aysmptoten>2k} we have that solutions remain admissible for every $s$ and the second equation of \eqref{system} implies that $Z$ is increasing for every $s\in \mathbb{R}$. We are only left to study the behavior as $s\to \infty$.

  \begin{pro}\label{pro:rhonegative}
 Assume that $\rho<0$ and $(X,Z)$ is a solution to \eqref{system} given by Theorem \ref{existence near the origin}.
 Then the solution exists for every $s\in \mathbb{R}$ and 
 $$ \lim_{s\to \infty} X=\gamma^k, \hbox{ and }  \lim_{s\to \infty} Z= \infty.$$
 
 \end{pro}
 
We  conclude the result from the following lemma.

\begin{lem}
	\label{X_Monotone}
		Let $(X,Z)$ be a solution of \eqref{system}  given by Theorem \ref{existence near the origin}.
		Assume that  $\rho<0$ and $n>2k$, then the function $X(s)$ is increasing for every $s\in\mathbb{R}$.
\end{lem}
\begin{proof}
Recall that when $\rho\leq 0$ holds $X< X_B$ in the admissible region, hence $Z_s>0$ (from the second equation of \eqref{system}). From the first equation of  \eqref{system} we note that if   $\overline{s}$ is a critical point of $X$ (that is $X_{s}(\overline{s})=0$) then  $$X_{ss}(\overline{s})=Z_{s}(\overline{s})f(X^{1/k}(\overline{s}))>0,$$
as long as $(X(\overline{s}),Z(\overline{s}))$ is in the admissible region. Therefore all critical points of $X$ are minima.\\

Since $$\lim_{s\to-\infty}X(s)=\lim_{s\to-\infty}Z(s)=0,$$ we necessarily have that $X$ is increasing for small $s$ and from the previous computation, it remains increasing.
\end{proof}

\begin{remark}\label{increasingX}
The previous proof can be extended to show that if $(X(s), Z(s))$ is a solutions to \eqref{system} that emanates from the origin then $X(s)$ remains increasing while $X(s)\leq X_B$, independently of the  values of $\rho$ and $n-2k$ . From the second equation in \eqref{system} we also have that $Z(s)$ remains increasing while $X(s)\leq X_B$.
\end{remark}

 \begin{proof}[Proof of Proposition \ref{pro:rhonegative}]
 
From Proposition \ref{pro:aysmptoten>2k} we have that  $$Z(s)\geq Z(s_0) e^{-\frac{k\rho}{\theta}(s-s_0)}>0 \hbox { for } s\geq s_0.$$

 In particular, since $\rho<0$, we have  $Z(s)\to \infty$ as $s\to \infty$.\\

  From Lemma \ref{X_Monotone} and Proposition \ref{pro:aysmptoten>2k}  we have that $X(s)$ converges as $s\to \infty$. Let $$X_\infty:=\lim_{s\to \infty} X(s)\leq \gamma^k.$$
 Assume that $X_\infty<\gamma^k$, then $F(X,Z)\to \infty$ as $s\to \infty$, where $F$ is defined by \eqref{definition_F_G}. In particular, there is a $\delta>0$ such that $X_s>\delta$ for $s>0$.
  This implies $X\geq X_0+\delta s$, which contradicts that $X$ is bounded. We conclude the desired result. 
 
 \end{proof}

\subsubsection{$k$-Yamabe steady solitons $\rho=0$.} 
Note that $X_B$ is not in the admissible region for $\rho=0$.
From Remark \ref{increasingX} and the second equation of \eqref{system} we have that 
 $X$ and $Z$ are monotonically increasing for every $s\in \mathbb{R}$. Since $X(s)< \gamma^k$, we can conclude that $X(s)$  is convergent to a finite value as $s\to \infty$.  \\ 
 
 Assume by contradiction that $\lim_{s\to \infty} X(s) \ne\gamma^k$. From the second equation in \eqref{system} this would imply that $$(\ln Z)_s\geq 2k\left(1-\frac{2k}{n+2k}X_\infty\right)>0.$$
In particular, $Z\to \infty$ as $s\to \infty$ and we can conclude as in the proof of Proposition \ref{pro:rhonegative} that $\lim_{s\to \infty} X(s) =\gamma^k$.\\

Assume now that $\lim_{s\to \infty} X(s) =\gamma^k$. If  $\lim_{s\to \infty} Z(s)=Z_\infty<\infty$ we would have that $$\lim_{s\to \infty} X_s(s)=-(n-2k) \gamma^k\left(1-\frac{k}{n+2k}\gamma\right)<0,$$ which contradicts that  $\lim_{s\to \infty} X(s) =\gamma^k$ (in fact, $X$ would not remain increasing). We conclude the following.
 
 \begin{pro}\label{pro:rhozero}
 Assume that $n>2k$, $\rho=0$ 
  and $(X,Z)$ is a solution to \eqref{system} given by Theorem \ref{existence near the origin}. Then the solution exists for every $s\in \mathbb{R}$ and 
 $$ \lim_{s\to \infty} X=\gamma^k, \hbox{ and }  \lim_{s\to \infty} Z= \infty.$$
 
 \end{pro}

\subsubsection{$k$-Yamabe shrinker soliton for $0<\rho \leq 2\theta$.} 

From Subsection \ref{subsec:linearization} we have the existence of solutions emanating from the origin. Moreover, the condition $0<\rho \leq 2\theta$ is equivalent to
$X_B\leq \gamma^k\leq X_A$ (and $A\not \in \mathcal{A}$). From Subsection \ref{critical_points} we have that the critical point $B$ is within the admissible region, hence we expect orbits to be of type $B$.
More precisely, we show the following.

\begin{pro}\label{pro:rhopositive}
 Assume that $0<\rho\leq 2 \theta$
  and $(X,Z)$ is a solution to \eqref{system} given by Theorem \ref{existence near the origin}.
 Then the solution exists for every $s\in \mathbb{R}$ and there exist positive constants $c$,  $C$ and $\delta$ such that
 $$ 0<c<   Z (s)<C\hbox{ and } X(s)<\gamma^k-\delta.$$
These orbits are of (generalized) type $B$.
 \end{pro}
We divide the proof into several intermediate lemmas.

\begin{lem}\label{Z_finite_positive}
 Assume that  $\rho>0$
 and $(X,Z)$ is a solution to \eqref{system} that satisfies the conditions of Proposition \ref{pro:rhopositive}, then
 $\sup_{s\in \mathbb{R}} Z(s) $ is  finite and positive.
 \end{lem}
\begin{proof}
From Remark \ref{increasingX} we have that $X$ remains increasing as long as $X(s)\leq X_B$. In this range $Z$ remains increasing as well.

There are two possibilities: Either $X(s)\leq X_B$ for every $s\in \mathbb{R}$ or there is a first $\bar{s}$ such that $X(\bar{s})= X_B$ and $X(s)> X_B$ for $\bar{s}<s<\bar{s}+\delta$ (for some $\delta>0$). In the first case, since both $X$ and $Z$ are monotonically increasing, we have that $(X(s), Z(s))$ converges and arguing as in the previous sections, we have that the limit is necessarily a critical point of the system (that cannot be the origin, since the functions are increasing). Then, the only possible limit is the point $B$
 and $Z\leq Z_B$.

If the orbit is not monotonically convergent, we have that $Z$ is decreasing while $X(s)>X_B$. Recall that we label $\bar{s}$ the first value of $s$ such that $X(\bar{s})= X_B$ and    if $X$ remains increasing then clearly $\sup_{s\in \mathbb{R}} Z(s)= Z(\bar{s}) $. Otherwise (that is $X$ has a critical point), the orbit cannot self intersect and $Z$ remains bounded by the values before this critical point (see Figure \ref{figure_Convergent_B}). Equivalently,  we have that $Z(s)\leq Z(\bar{s})$, where $\bar{s}$ is the smallest value for which  $X(\bar{s})= X_B$.

\end{proof}

 \begin{lem}\label{convergentcasewhopositive}
 Assume that $0<\rho\leq 2\theta$
and $(X,Z)$ is a solution to \eqref{system} given by Theorem \ref{existence near the origin}.  If $\lim_{s\to \infty} X(s)=X_\infty <\infty$ and 
 $\lim_{s\to \infty} Z(s)=Z_\infty<\infty$, then $$ \lim_{s\to \infty} (X(s), Z(s))=B.$$
 
 \end{lem}
 
\begin{proof}
Observe first that since $\gamma^k\leq X_A$, Proposition \ref{pro:aysmptoten>2k} applies to this case, namely
$X(s)\leq\gamma^k$ and $Z(s)\geq Z(s_0) e^{-\frac{k\rho}{\theta} (s-s_0)} $ for every $s\geq s_0$. Note in addition that $$ \lim_{s\to \infty} (X(s), Z(s))\ne (0,0),$$ since $Z$ is increasing for $X<X_B$.

Assume first that  $\lim_{s\to \infty} X(s)=X_\infty <\infty$, $X_\infty\ne X_B$ and $Z_\infty \ne 0$. Then either $Z_s\to \delta>0$ (if  $X_\infty < X_B$), in which case $Z\to \infty$, contradicting Lemma \ref{Z_finite_positive} or $ Z_s\to -\delta<0$, contradicting that $Z$ remains positive. If $Z_\infty =0$, then necessarily $X_\infty > X_B$ (since otherwise $Z$ is increasing) and using \eqref{system} we have $$\lim_{s\to \infty}X_s= -(n-2k)\left(1-\frac{k}{n+2k}X_\infty^{1/k}\right)X_\infty<-\delta.$$ This would imply that $X$ is not convergent, contradicting the hypotheses.
 We conclude that $X_\infty= X_B$.

To finish the proof, we need to show that
$\lim_{s\to \infty} Z(s)=Z_B$. Observe that if $X_s\to 0$, then $Z\to Z_B$. Since we already know that $X_s$ converges (because the right hand side of \eqref{system} is convergent by hypothesis), then necessarily $X_s\to 0$ (since otherwise, $X$ would not converge to a finite value). We conclude the desired result.

\end{proof}

\begin{lem}\label{criticalpointsrhopositve}
Let $0<\rho\leq 2\theta$
 and $(X,Z)$ be a solution to \eqref{system} given by Theorem \ref{existence near the origin}. Assume that $ \lim_{s\to \infty} (X(s), Z(s))\ne B$  and that there is a $\bar{s}\in \mathbb{R}$ such that $X(\bar{s})>X_B$.
 Then there are values $x_1 < x_2$ that are critical points of $X$ (that is $X_s(x_1)=X_s(x_2)=0$) such that
  $$ X(x_2)<X(s)<X(x_1) \hbox{ for every }s>x_2.$$ 
  
  Similarly,  there are values $z_1 < z_2$ that are critical points of $Z$ (that is $Z_s(z_1)=Z_s(z_2)=0$) $$ Z(z_2)<Z(s)<Z(z_1) \hbox{ for every }s>z_2. $$
 Moreover, $z_1<x_1 <z_2< x_2$.

 \end{lem}
\begin{proof}

We depict the argument in the Figure \ref{figure_Convergent_B} and we recommend the reader to keep this picture in mind while reading the proof below.\\

	\begin{figure}[h]
	\centering

\tikzset{every picture/.style={line width=0.75pt}} %set default line width to 0.75pt        

\begin{tikzpicture}[x=0.75pt,y=0.75pt,yscale=-1,xscale=1]
	%uncomment if require: \path (0,310); %set diagram left start at 0, and has height of 310
	
	%Curve Lines [id:da27028950599899004] 
	\draw [color={rgb, 255:red, 208; green, 2; blue, 27 }  ,draw opacity=1 ][line width=1.5]    (108,239.6) .. controls (106,240.6) and (195.17,89.82) .. (256,57.6) .. controls (316.83,25.38) and (353.06,44.25) .. (372,65.6) .. controls (390.94,86.95) and (390,110.6) .. (388,122.6) .. controls (386,134.6) and (381.2,157.32) .. (355,172.6) .. controls (328.8,187.88) and (291,180.6) .. (274,163.6) .. controls (257,146.6) and (252,115.6) .. (267,94.6) .. controls (282,73.6) and (316,60.6) .. (339,77.6) .. controls (362,94.6) and (362,122.6) .. (354,135.6) .. controls (346,148.6) and (334.02,154.06) .. (324,155.6) .. controls (313.98,157.14) and (296,149.6) .. (291,142.6) .. controls (286,135.6) and (283.86,123.56) .. (286,115.6) .. controls (288.14,107.64) and (292,104.6) .. (296,99.6) ;
	%Straight Lines [id:da7843945611548797] 
	\draw    (94,241.6) -- (464,240.61) ;
	\draw [shift={(466,240.6)}, rotate = 179.85] [color={rgb, 255:red, 0; green, 0; blue, 0 }  ][line width=0.75]    (10.93,-3.29) .. controls (6.95,-1.4) and (3.31,-0.3) .. (0,0) .. controls (3.31,0.3) and (6.95,1.4) .. (10.93,3.29)   ;
	%Straight Lines [id:da9417938537219737] 
	\draw    (107,248) -- (107.99,4.6) ;
	\draw [shift={(108,2.6)}, rotate = 90.23] [color={rgb, 255:red, 0; green, 0; blue, 0 }  ][line width=0.75]    (10.93,-3.29) .. controls (6.95,-1.4) and (3.31,-0.3) .. (0,0) .. controls (3.31,0.3) and (6.95,1.4) .. (10.93,3.29)   ;
	%Shape: Circle [id:dp6873199888573873] 
	\draw  [fill={rgb, 255:red, 0; green, 0; blue, 0 }  ,fill opacity=1 ] (317,39.8) .. controls (317,38.03) and (318.43,36.6) .. (320.2,36.6) .. controls (321.97,36.6) and (323.4,38.03) .. (323.4,39.8) .. controls (323.4,41.57) and (321.97,43) .. (320.2,43) .. controls (318.43,43) and (317,41.57) .. (317,39.8) -- cycle ;
	%Shape: Circle [id:dp1742998468559982] 
	\draw  [fill={rgb, 255:red, 0; green, 0; blue, 0 }  ,fill opacity=1 ] (385,116.8) .. controls (385,115.03) and (386.43,113.6) .. (388.2,113.6) .. controls (389.97,113.6) and (391.4,115.03) .. (391.4,116.8) .. controls (391.4,118.57) and (389.97,120) .. (388.2,120) .. controls (386.43,120) and (385,118.57) .. (385,116.8) -- cycle ;
	%Shape: Circle [id:dp10715324077025734] 
	\draw  [fill={rgb, 255:red, 0; green, 0; blue, 0 }  ,fill opacity=1 ] (255,117.8) .. controls (255,116.03) and (256.43,114.6) .. (258.2,114.6) .. controls (259.97,114.6) and (261.4,116.03) .. (261.4,117.8) .. controls (261.4,119.57) and (259.97,121) .. (258.2,121) .. controls (256.43,121) and (255,119.57) .. (255,117.8) -- cycle ;
	%Shape: Circle [id:dp8691031389694965] 
	\draw  [fill={rgb, 255:red, 0; green, 0; blue, 0 }  ,fill opacity=1 ] (319,180.8) .. controls (319,179.03) and (320.43,177.6) .. (322.2,177.6) .. controls (323.97,177.6) and (325.4,179.03) .. (325.4,180.8) .. controls (325.4,182.57) and (323.97,184) .. (322.2,184) .. controls (320.43,184) and (319,182.57) .. (319,180.8) -- cycle ;
	%Shape: Circle [id:dp051663449616187185] 
	\draw  [fill={rgb, 255:red, 0; green, 0; blue, 0 }  ,fill opacity=1 ] (317,117.8) .. controls (317,116.03) and (318.43,114.6) .. (320.2,114.6) .. controls (321.97,114.6) and (323.4,116.03) .. (323.4,117.8) .. controls (323.4,119.57) and (321.97,121) .. (320.2,121) .. controls (318.43,121) and (317,119.57) .. (317,117.8) -- cycle ;
	%Curve Lines [id:da14556232203455055] 
	\draw     ;
	%Curve Lines [id:da9446649612938665] 
	\draw [color={rgb, 255:red, 208; green, 2; blue, 27 }  ,draw opacity=1 ][line width=1.5]  [dash pattern={on 1.69pt off 2.76pt}]  (296,99.6) .. controls (318,79.6) and (334,96.6) .. (335,94.6) ;
	
	% Text Node
	\draw (89,3) node [anchor=north west][inner sep=0.75pt]   [align=left] {$Z$};
	% Text Node
	\draw (468,243.6) node [anchor=north west][inner sep=0.75pt]   [align=left] {$X$};
	% Text Node
	\draw (315,96.8) node [anchor=north west][inner sep=0.75pt]   [align=left] {$B$};
	% Text Node
	\draw (185,126) node [anchor=north west][inner sep=0.75pt]   [align=left] {$(X,Z)(x_{2} )$};
	% Text Node
	\draw (389.2,124) node [anchor=north west][inner sep=0.75pt]   [align=left] {$(X,Z)(x_{1} )$};
	% Text Node
	\draw (285,16) node [anchor=north west][inner sep=0.75pt]   [align=left] {$(X,Z)(z_{1} )$};
	% Text Node
	\draw (288,188) node [anchor=north west][inner sep=0.75pt]   [align=left] {$(X,Z)(z_{2} )$};

\end{tikzpicture}
\caption{}
\label{figure_Convergent_B}
\end{figure}

Note first that since there is a $\bar{s}$ such that $X(\bar{s})>X_B$, then there is a first $s:=z_1$ such that $X(z_1)=X_B$. From \eqref{system} we have that $Z_s(z_1)=0$ and $Z$ is decreasing for $s>z_1$ while $X$ is increasing. Since $X\leq \gamma^k$, if $X$ remains increasing, we would have that $X$ and $Z$ converge, but Lemma \ref{convergentcasewhopositive} implies that the limit is $B$, which is not possible ($X(s)>X_B$ and increasing). Then
 there is a first value of $s:=x_1>z_1$ such that 
$X_s(x_1)=0$.  Moreover, $X_s(s)>0$ for $s<x_1$ and since $X(z_1)=X_B$ we have that $X(x_1)>X_B$, $Z_s(x_1)<0$. From the first equation in  \eqref{system} we have at the critical point that
$$X_{ss}(x_1)= Z_s(x_1) f(X^{\frac{1}{k}})<0.$$

Hence it is a local maximum and $X$ cannot have any other critical point while $X>X_B$. Note that if $X(s)>X_B$ for $s>x_1$, then $X$ and $Z$ remain monotone and bounded, hence they converge and  Lemma \ref{convergentcasewhopositive} implies that $ \lim_{s\to \infty} (X(s), Z(s))= B$. Hence, from the hypotheses, we conclude that there is a $z_2$ for which $X(z_2)=X_B$ and $Z(z_2)\ne Z_B$. This implies (from \eqref{system}) that $Z_s(z_2)=0$  and $Z$ becomes increasing for $z_2<s<z_2+\delta$ (for some $\delta>0$). In particular, $z_2$ is a local minimum of $Z$. Using the same argument we have that  if $X(s)<X_B$ for $s>z_2$, then $X$ and $Z$ remain monotone and bounded, hence convergent (to $B$) and this would contradict the hypotheses. Then necessarily, there is $x_2>z_2$ such that $X(x_2)<X_B$, $X_s(x_2)=0$ and $$X_{ss}(x_2)= Z_s(x_2) f(X^{\frac{1}{k}})>0.$$

Since the orbit cannot self-intersect, then for $s>x_2$ necessarily the trajectory stays below the curve $(X(s), Z(s))$ with $s\leq z_1$ while $X(s)<X_B$. In particular, this implies $ Z(s)<Z(z_1)$ for every $s>z_1$. With a similar argument, we have that $ Z(z_2)<Z(s)$ when $s>z_2$. Using the same reasoning we may conclude that $X(s)<X(x_1) \hbox{ for every }s>x_1$ and $X(s)>X(x_2) \hbox{ for every }s>x_2,$ concluding the result.

\end{proof}

\begin{proof}[Proof of Proposition \ref{pro:rhopositive}]

From Lemma \ref{criticalpointsrhopositve} we have that the orbit remains in the region where the functions on the right-hand side of \eqref{system} are smooth and differentiable, hence the orbit can be continued for every $s\in \mathbb{R}$. If the orbit converges, from Lemma \ref{convergentcasewhopositive} we have that it converges to $B$ and the result follows (since $Z_B>0$). Otherwise,  choosing the constants   $c=Z(z_2)$, $C=Z(z_1) $ and $0<\delta<\gamma^k-X(x_1)$ from  Lemma \ref{criticalpointsrhopositve}, we can conclude the result.

\end{proof}

\begin{remark}

Standard results of dynamical systems \cite{DynamicalSystem_DanHenry, DynamicalSystem_Gerald}  imply that the orbit either converges to a critical point, which is necessarily $B$ by Lemma \ref{convergentcasewhopositive} or to a periodic orbit, which due to Lemma \ref{criticalpointsrhopositve}, would enclose $B$ but would remain away from $Z=0$. We expect the latter case not to happen, but this would need an additional proof.

\end{remark}

\subsubsection{$k$-Yamabe expanders soliton for $2 \theta <\rho $.} 

In contrast with the previous cases, $X_A< \gamma^k$ and the difficulty resides in showing that the orbit remains admissible.  We build a barrier for the solution to ensure that this is the case. More precisely, following the proof in Section \ref{subsec:linearization} we can show that System \eqref{system_A} has a local solution that converges to $A$ as $s\to \infty$ and  the following holds.\\

 \begin{thm}\label{existence near A}
Assume $\rho\ge 2\theta$. For every $\bar{\alpha}>0$ there is an $s_0\in \mathbb{R}$ such that a solution $(W,V)$ to \eqref{system_A} 
   exists for $s\in[ s_0, \infty)$ and  satisfies
 \begin{align*}
 \lim_{s\to\infty}W(s)&=\lim_{s\to \infty}V(s)=0,\\
  \lim_{s\to\infty}e^{2ks}W(s)&=\frac{n}{h(0)}  \lim_{s\to\infty}e^{2ks}V(s)=\bar{\alpha},\\
	W(s)>0,  & \quad V(s)>0.\
\end{align*}

\end{thm}

\begin{remark}\label{Wmonotone}
Let  $(W,V)$ be the orbit given by Theorem \ref{existence near A}.
Note that while  $(W,V)(s)\in \mathcal{A}$, standard ODE theory implies that  the orbit can be continued (backwards). Moreover, following the proof of Proposition \ref{X_Monotone} and Remark \ref{increasingX} we can show that  $W$ and $V$ are monotonically decreasing while $W(s)\leq X_B.$

From these observations and System \eqref{system_A}, it is direct to show that there is a finite $w_B$ such that $W(w_B)=X_B$ and   $W(s)$ and $V(s)$ are monotonically decreasing for $s\geq w_B$.
\end{remark}

\medskip

Using the solution of Theorem \ref{existence near A} as barrier for the solution from Theorem \ref{existence near the origin}  we have the following result.

\begin{pro}\label{pro:rhopositivebiggerthat2theta}
 Assume that $\rho>2 \theta$, 
   $(X,Z)$ is a solution (to \eqref{system}) given by Theorem \ref{existence near the origin} and
   $(W,V)$ is a solution (to \eqref{system_A}) given by Theorem \ref{existence near A}.
For $X(s)\geq X_B\geq W(w)$ and $W^{1/k}(w)=X_A^{1/k}- X^{1/k}(s)$ holds
  $$Z(s)\leq V(w).$$
 \end{pro}

\begin{proof}

Note first that, since $(W,V)$ represents a solution to \eqref{system} under a change of variables, both orbits (the ones defined by  $(X,Z)$ and $(W,V)$)
cannot intersect, unless they coincide. Then, to conclude the result  is enough to show that $Z(s_B)\leq V(w_B)$, where $s_B$ is the first value of $s$ at which $X(s)=X_B$ and $w_B$ is given in Remark \ref{Wmonotone}. More precisely, we would finish the proof as follows: If  $Z(s_B)= V(w_B)$ then they represent the same solution and it is an orbit of type $A$. If $Z(s_B)< V(w_B)$ the orbit $(X,Z)$  remains below the orbit defined by $(W,V)$  and hence, it is admissible for every $s\in \mathbb{R}$. Arguing as in the case $\rho \leq 2\theta$, we have that if  $(X,Z)$ converges when $s\to \infty$ then the limit is a critical point of the system, that can either be $A$ or $B$.  If it does not converge, Lemma \ref{criticalpointsrhopositve} can be applied with minor modifications to conclude that the orbit is of generalized type $B$.

We devote the rest of the proof to show $Z(s_B)\leq V(w_B)$ (where $s_B$ is the first value for which  $X(s_B)=X_B$ and $w_B$ is the last value for which $ W(w_B)=X_B$).

We  compare $(X,Z)$ and  $(W,V)$ as follows: we show that if $X(s)= W(w)$ (with $s\leq s_B$ and $w\geq w_B$) then $Z(s)\leq V(w)$.
To achieve this, we  first change the orientation of the curve $(W(w),V(w))$, namely we
define $W_-(w)=W(-w)$ and  $V_-(w)=V(-w)$. Then we have from \eqref{system_A} that
\begin{align*}
		\left\{
	\begin{gathered}
		(W_-)_{w}=-(n-2k)W_-\left(1-\frac{k}{n+2k}W_-^{1/k}\right)+  V_- h(W_-^{\frac{1}{k}}),\hfill\\
		(V_-)_{w}=2kV_-\left(1-\frac{2k}{n+2k}W_-^{1/k}\right),\hfill
	\end{gathered}
	\right.
\end{align*}

Now we observe that
 since $X$ is an increasing function of $s$ while $X(s)\leq X_B$ (from Remark \ref{X_Monotone}), $X$ is invertible and we can write $s$ as a function of $X$. Then we have
 $Z=Z(X)$ with $X\in [0, X_B]$ and 
  \begin{equation} \label{derZX}\frac{dZ}{dX}= \frac{2kZ\left(1-\frac{2k}{n+2k}X^{1/k}\right)}{-(n-2k)\left(1-\frac{k}{n+2k}X^{1/k}\right)X+Z\, f(X^{1/k})}.\end{equation}
 
 Similarly,  $W_-$ is increasing for $W_-\leq X_B$ (from Remark \ref{Wmonotone} and the change of orientation) and we can write $V_-=V(W_-)$. We also have
 \begin{equation} \label{derVW} \frac{dV_-}{d W_-}=  \frac{2kV_-\left(1-\frac{2k}{n+2k}W_-^{1/k}\right)}{-(n-2k)W_-\left(1-\frac{k}{n+2k}W_-^{1/k}\right)+  V_- h(W_-^{\frac{1}{k}})}..\end{equation}

Since we are interested in comparing $Z$ and $V_-$ when $X= W_-$, with an abuse of language we write $V_-(X)$.
When $X=0$  we have that $$\frac{dZ}{dX}=\lim_{s\to -\infty} \frac{Z}{X}=\frac{n}{f(0)}, \hbox{ and }$$
 $$\frac{dV_-}{d X}=\lim_{s\to -\infty} \frac{V_-}{W_-}=\lim_{s\to \infty} \frac{V}{W}=\frac{n}{h(0)}.$$
The constant  $\nu$ in \eqref{defh}  is given by
$\nu=\gamma -X_A^{\frac{1}{k}}<\gamma$, 
where $\gamma$ the constant in  \eqref{function_f}. 
Since $f(0)=c_{n,k} \beta^k \gamma^k$ and $h(0)=c_{n,k} \beta^k \nu^k$ we conclude
$$\frac{dZ}{dX}(0)\leq \frac{dV_-}{dX}(0).$$
Observing that $Z(0)=V_-(0)=0$ we have
 for small values of $X$ that
 $$ Z(X)<V_-(X).$$
In addition, if $0<X^{1/k}<\frac{X^{\frac{1}{k}}_A}{2}= X^{\frac{1}{k}}_B$ we have
$$\nu+X^{1/k}= \gamma- X^{\frac{1}{k}}_A+X^{1/k} <\gamma-X^{1/k}.$$
Then \begin{align}\label{compfh}f(X^{1/k})&=c_{n,k}\beta^{k}\left(1-\frac{k}{n+2k}X^{1/k}\right)\left(\frac{\gamma-X^{1/k}}{1-\frac{k}{n+2k}X^{1/k}}\right)^{k}\\&>
c_{n,k}\beta^{k}\left(1-\frac{k}{n+2k}X^{1/k}\right)\left(\frac{\nu+X^{1/k}}{1-\frac{k}{n+2k}X^{1/k}}\right)^{k}=h(X^{1/k}). \notag
\end{align}
Assume there is an $X\in (0, X_B)$ such that $Z(X)=V_-(X)$. From \eqref{compfh}, \eqref{derZX} and \eqref{derVW}we have at this value of $X$ that $$ \frac{dZ}{dX}(X)\leq  \frac{dV_-}{dX}(X).$$
In particular, $Z-V_-$ decreases and we can conclude $$Z(X)\leq V_-(X)\hbox{ for }X\leq X_B,$$
finishing the proof.

\end{proof}

\subsection{Existence of orbits in the case $n=2k$}
\label{orbitsn=2k}
In this case the system \eqref{system} is reduced to 	
	\begin{equation}
	\label{system_n=2k}
	\left\{
	\begin{gathered}
		X_{s}=Z\, f(X^{1/k})\hfill\\
		Z_{s}=2kZ\left(1-\frac{1}{2}X^{1/k}\right),\hfill
	\end{gathered}
	\right.
\end{equation}
with

	\begin{equation*}
		f(X^{1/k})=c_{n,k}\beta^k\left(1-\frac{1}{4}X^{1/k}\right)\left(\frac{\gamma-X^{1/k}}{1-\frac{1}{4}X^{1/k}}\right)^{k},
	\end{equation*}
	with $\gamma=\frac{2\theta +\rho}{\theta}.$
	
	The existence of orbits emanating from the origin follows from Theorem \ref{existence near the origin} and the behavior for $\rho\leq 0$  is analogous to the  case $n>2k$.
	The proofs in Subsection \ref{solitonsn>2k} can be performed in this case, obtaining analogous results for most propositions, except when $\rho=0$, where the proof of  Subsection \ref{solitonsn>2k} only shows that $Z\geq \delta$ as $s\to \infty$. We complete the proof of Theorem \ref{classification} as follows.

	 \begin{pro}\label{pro:rhonegativen=2k}
 Assume that $\rho\leq0$ and $(X,Z)$ is a solution to \eqref{system} given by Theorem \ref{existence near the origin}.
 Then the solution exists for every $s\in \mathbb{R}$ and 
 $$ \lim_{s\to \infty} X=\gamma^k, \hbox{  }  \lim_{s\to \infty} Z= \infty.$$
 
 \end{pro}
\begin{proof}	
The proof when $\rho<0$ is identical to the one of Proposition \ref{pro:rhonegative}. We can also prove as in Proposition \ref{pro:rhozero} that for $\rho=0$ necessarily $\lim_{s\to \infty}X(s)=\gamma^k$. We only need to show that $Z\to \infty.$
Since $\rho=0$ and $n=2k$ we have that $\gamma=2$.
Recalling the definition of $f$ in \eqref{function_f} and combining both equations  in  \eqref{system_n=2k} we have for $k\ne 2$ that
\begin{equation}((2-X^{\frac{1}{k}})^{-k+2})_s= \frac{(k-2) c_{n,k}}{2k^2}\beta^k X^{1/k-1} Z_s \left(1-\frac{1}{4}X^{1/k}\right)^{1-k}. \label{relxzrho=0}\end{equation}

For $s$ large enough we may assume that $1\leq X(s)\leq 2^k$. Hence there is a positive constant $C$ such that
$$ (2-X^{\frac{1}{k}})^{-k+1}(s)-  (2-X^{\frac{1}{k}})^{-k+1}(s_0)\leq C(Z(s)- Z(s_0)).$$
Since $k>2$ and $X^{\frac{1}{k}}(s)\to 2=\gamma$ as $s\to \infty$ we have that $Z(s)\to \infty$. 

For $k=2$ we obtain a similar result by taking
$$(\ln (2-X^{\frac{1}{2}}))_s= C  X^{-1/2} Z_s \left(1-\frac{1}{4}X^{1/2}\right)^{-1},$$ for some constant $C$. As before, integrating we have (for a different constant $C$) that
$$ \ln (2-X^{\frac{1}{2}})(s)-  \ln(2-X^{\frac{1}{2}})(s_0)\leq C(Z(s)- Z(s_0)),$$
and we conclude as above.

\end{proof}

\subsubsection{$k$-Yamabe expanders soliton for $0<\rho\leq 2 \theta  $.} 

In this case we prove the following

\begin{pro}\label{pro:rhopositive=2k}
Let $\rho\leq 2 \theta$ and
   $(X,Z)$ be a solution to \eqref{system} given by Theorem \ref{existence near the origin}.
 Then the solution $(X,Z)$ exists for every $s\in \mathbb{R}$ and is an orbit of (generalized) type $A$.
 \end{pro}

\begin{proof}
We recall from Subsection \ref{critical_points} that $B$ is not a critical point in this case. Moreover, since $X$ remains increasing and bounded by $\gamma^k$ (from Proposition \ref{pro:aysmptoten=2k}),
we have that there is a finite value of $s=s_B$ such that $X(s_B)=X_B$ and $X(s)>X_B$ for $s>s_B$. $Z$ is decreasing for $s>s_B$ and hence bounded  from above by $Z(s_B)$. In addition, we have that
 $$(\ln Z)_s\geq n\left(1-\frac{1}{2}\gamma\right).$$
 Hence $Z$ is uniformly bounded from below for finite values of $s$ by $$Z(s)>Z(s_0) e^{n\left(1-\frac{1}{2}\gamma\right)s}>0.$$

 We conclude that $X$ and $Z>0$ are monotone for $s>s_B$, bounded and hence convergent. As in the previous cases, the orbit must converge to a critical point, which in this case has to be of the form $(X_\infty, 0)$ where $X_B<X_\infty\leq \gamma^k$.
 
\end{proof}

\subsubsection{$k$-Yamabe expanders soliton for $2 \theta <\rho $.}

We conclude this section by showing

 \begin{pro}\label{pro:rhopositivebiggerthat2thetan=2k}
Let $\rho>2 \theta$ and 
   $(X,Z)$ a solution to \eqref{system} given by Theorem \ref{existence near the origin}.
 Then  the solution $(X,Z)$ exists for every $s\in \mathbb{R}$ and is an orbit of (generalized) type $A$.
 \end{pro}

\begin{proof}
As in  Proposition \ref{pro:rhopositivebiggerthat2theta} we have that
$(W,V)$ given by Theorem \ref{existence near A} acts as barrier of $(X,Z)$  when $X\geq X_B$. In particular, $X(s)\leq X_A$ and monotone. On the other hand,  $Z$ is bounded by the value that attains when $X=X_B$ and decreasing if $X\geq X_B$. We conclude that the orbit converges to a critical point of the form
$(X_\infty, 0)$ where $X_B<X_\infty\leq X_A$.

\end{proof}

\subsection{Existence  and non-existence of orbits in the case $n<2k$ (proof of Theorem \ref{classification3})}

The existence of orbits emanating from $(0,0)$  is easier to show in this case, since the origin is a stable source.   However, in order to have a good control at the origin and for large values of $s$  we consider the solutions provided by Theorem \ref{existence near the origin}.  We remark, however, that there may be other solutions in this case.\\

 Note that
$X_s>0$ in the admissible region and  for $\delta>0$ small enough we have that if
 $\delta< X(s)<X_A-\delta$ then  $X_s> (2k-n) \delta^2>0$. In particular, for all cases that $\rho< 2\theta $ (including $\rho<0$), we have that if  $X(s_0)=\delta$, then for 
$s>\frac{\gamma^k-\delta}{(2k-n) \delta^2}$ holds $X(s)> \gamma^k$ and the solution is non-admissible for those values of $s$. We observe that this statement holds true for any solution emanating from the origin (not only the ones given by Theorem \ref{existence near the origin}).  \\

The only possible admissible orbits occur for $ \rho\geq 2\theta $ and in fact this can be shown following exactly the proof of Proposition \ref{pro:rhopositivebiggerthat2theta}. Since $B$ is not a critical point in this case (because $X$ is always strictly increasing for $X\leq X_A$), we conclude the following result, that finishes the proof of Theorem \ref{classification3}.

\begin{pro}\label{pro:rhopositiven<2k}
 Assume that $\rho\geq 2 \theta$, 
   $(X,Z)$ is a solution to \eqref{system} given by Theorem \ref{existence near the origin}.
 Then  the solution $(X,Z)$ exists for every $s\in \mathbb{R}$ and is an orbit  of type $A$. 
 \end{pro}

\section{Asymptotic behavior}\label{sec:asymtoticbehavior}\text{ } 
In this section we focus on studying the behavior of our solutions near the origin and at infinity, namely in proving  Theorem \ref{classification3}. We also verify that the solution $u$  is strictly positive at the origin.\\

\subsection{The behavior near the origin}
Recalling the definition of $Z$ in \eqref{variables} and Theorem \ref{existence near the origin} We have that
$$\lim_{r\to 0} u^{(1-m)k}(r)=\lim_{s\to -\infty} e^{-2ks} Z(s)=\alpha_k>0.$$
In addition we have
\begin{align*}
	\frac{Z}{X}=\left(-\frac{ru^{2-m}}{u_{r}}\right)^{k}\to \frac{n}{f(0)}=\frac{2^{k}n\binom{n-1}{k-1}\left(\frac{k}{n+2k}\right)^{k-1}}{(n+2)(2\theta+\rho)^{k}}.
\end{align*}
From here that $$ \frac{u_{r}}{u^{m-2}}\sim  -\left(\frac{f(0)}{n}\right)^{1/k}r.$$
Integrating the previous equation we obtain %{\textcolor{magenta}{la constante $C_{n,k}$ no la antecede un signo menos...}}
\begin{align*}
	u^{3-m} (r)\sim \alpha^{3-m}
	- \frac{1}{2}\left(\frac{f(0)}{n}\right)^{1/k}r^2  \hbox{ as } r\to 0, 
\end{align*}
where $3-m=\frac{2(n+4k)}{n+2k}>0$.

\subsection{Behavior near infinity when $n>2k$}
\subsubsection{$k$-Yamabe expander soliton $\rho<0$.}

\medskip

From \eqref{lowerboundz} we have $$Z(s)\geq Ce^{-\frac{k \rho}{\theta}s}.$$ To prove the decay rate in Theorem \ref{classification} we  show that we have an upper bound for $Z$ of the same order.

In the proof of Proposition \ref{pro:rhonegative} we argued that
   $X_s\to 0$ as $s\to \infty$ (since otherwise $X$ would no converge). Then, from the first line of \eqref{system} and the definition of $f$ we have that
\begin{equation} Z(\gamma -X^{\frac{1}{k}})^k\to C> 0,\label{ZXrhonegative}\end{equation}
where $C=\lim_{X\to \gamma^k} \frac{n-2k}{c_n,k}\beta^k\left(1-\frac{k}{n+2k}X^{1/k}\right)^kX$. 

This implies that  for $s$ large enough $$ \gamma -X^{\frac{1}{k}}\leq 2 C Z^{-1/k}\leq2 C e^{\frac{\rho}{2\theta}s}.$$
%{\textcolor{magenta}{la potencia de la exponencial me da $\frac{\rho}{\theta}s$.}}
 Using the second equation of \eqref{system} this implies %{\textcolor{magenta}{la potencia a continuación me da $\frac{(k+1)\rho}{\theta}s$ y el coeficiente no sé de donde sale!!}}
$$(\ln Z)_s\leq 2k \left(1-\frac{2k}{n+2k}\gamma\right) + \frac{4k^2}{n+2k} C e^{\frac{\rho}{2\theta}s}.$$
Integrating we have
$$\ln \frac{Z (s)}{Z(s_0)}\leq 2k \left(1-\frac{2k}{n+2k}\gamma\right) (s-s_0)+\frac{(2k)^2 \theta}{\rho} C \left(e^{\frac{\rho}{2\theta}s}-e^{\frac{\rho}{2\theta}s_0}\right).$$

Recalling that $1-\frac{2k}{n+2k}\gamma=- \frac{\rho}{2\theta}$ and that $\rho<0$ (hence the second term is uniformly bounded) we have
$$Z(s)\leq Ce^{-\frac{k \rho}{\theta}s}.$$
Combining this inequality with  \eqref{lowerboundz} and using the definition of $Z$ and $s$ in \eqref{variables} we conclude
$$u (x)= \text{O}(|x|^{-\frac{2}{1-m}-\delta})\hbox{ as }|x|\to \infty,$$ where $\delta=\frac{k \rho}{(1-m)\theta}$.

\begin{remark}
Since $\rho+2\theta>0$ we have that $ -\frac{\rho}{2\theta}\leq 1$.
\end{remark}

\subsubsection{$k$-Yamabe steady soliton $\rho=0$.}

\medskip

Although \eqref{lowerboundz} holds in this case, it is not enough to estimates the growth of $Z$ as $s\to \infty$. We observe however that  when  $\rho=0$ holds $\gamma=\frac{n+2k}{2k}$, hence the second equation in \eqref{system} is equivalent to
$$2k(\gamma -X^{\frac{1}{k}})=\gamma \frac{Z_s}{Z}.$$
On the other hand, \eqref{ZXrhonegative} still holds (for  $C= \frac{n-2k}{\beta^k c_{n,k}}\left(\frac{n+2k}{2k}\right)^k>0$) and
 for $s$ large enough we have (for a different constant $C$) that 
$$\frac{C}{2}\leq \gamma \frac{Z_s}{Z^{1-1/k}}\leq 2C.$$
Integrating we have %{\textcolor{magenta}{cambié el $k$ que estaba como numerador y según mis cuentas queda dividiendo!}}
$$\frac{C}{2k}(s-s_0)\leq \gamma (Z^{1/k}(s)-  Z^{1/k}(s_0))\leq \frac{2C}{k} (s-s_0).$$
Equivalently, for $s$ large enough $Z(s)\sim s^k.$ Hence the definition of $Z$ and $s$ in \eqref{variables}  implies 			
	 $$u(x)= \text{O}\left(\left[\frac{\ln |x|}{|x|^2}\right]^\frac{1}{1-m}\right) \hbox{ as }|x|\to \infty.$$

\subsubsection{$k$-Yamabe shrinking soliton  $0<\rho $.} 

From Proposition \ref{pro:rhopositive}, for orbits of type $B$ we have that there are $c, C>0$ such that $$c\leq Z(s)\leq C.$$
Equivalently,
 $$u(x)=\text{O}(|x|^{-\frac{2}{1-m}}) \hbox{ as }|x|\to\infty.$$

For orbits of type $A$ (that may occur if $\rho>2\theta$) we have that $Z(s)\sim e^{-2k s}$ for $s$ large, or equivalently
$$u(x)=\text{O}(|x|^{-\frac{4}{1-m}}) \hbox{ as }|x|\to\infty.$$

\subsection{Behavior near infinity when $n=2k$}

\medskip

In this case some of the parameters are more explicit: $m=0$, $\gamma=\frac{2\theta+\rho}{\theta}$, $X_B^{\frac{1}{k}}=2$ and  $X_A^{\frac{1}{k}}=4$. We again analyze each case.

\subsubsection{$k$-Yamabe expander soliton $\rho<0$.}

We still have that \eqref{lowerboundz} holds. That is  $$Z\geq Ce^{-\frac{\rho}{2\theta}s}.$$
We show again that we have an upper bound of the same order in this case, but with a different argument than when $n>2k$.\\

We use \eqref{system_n=2k} to observe that
\begin{equation}((\gamma-X^{1/k})^{-k+1})_s =\frac{(k-1)c_{n,k}\beta^k}{2k^2} Z_s  X^{1/k-1}\frac{\left(1-\frac{1}{4}X^{1/k}\right)^{-k+1}}{1-\frac{1}{2}X^{1/k}}\label{relXZn=2krho<0}.\end{equation}
 and $\gamma= \left(\frac{2\theta+\rho}{\theta}\right)<2.$
 Since $X\to \gamma^k$ as $s\to \infty$ we have that 
 $$\lim_{s\to \infty}\frac{(k-1)c_{n,k}\beta^k}{2k^2}  X^{1/k-1}\frac{\left(1-\frac{1}{4}X^{1/k}\right)^{-k+1}}{1-\frac{1}{2}X^{1/k}}= 
 \frac{(k-1)c_{n,k}\beta^k\theta}{k^2 (-\rho)}  \left(\frac{4\theta^2-\rho^2}{4\theta^2}\right)^{1-k}> 0.
 $$
 Hence the term multiplying  $Z_s$ in Equation \eqref{relXZn=2krho<0} is bounded from above and below by a constant (that depends on the parameters).  Combining this inequality with \eqref{relXZn=2krho<0} and that $Z_s>0$ we have for $s_0$ sufficiently large that
  $$(\gamma-X^{1/k})^{-k+1}(s)- (\gamma-X^{1/k})^{-k+1}(s_0) \geq C(Z-Z_0) \geq  Ce^{-\frac{\rho}{2\theta}s}-CZ_0. $$
This implies,
$$(\gamma-X^{1/k})^{k-1}(s)\leq C  \frac{e^{\frac{\rho}{2\theta}s}}{1-Ce^{\frac{\rho}{2\theta}s}Z_0}\leq 
C e^{\frac{\rho}{2\theta}s}.$$
The constant $C$ denotes a generic quantity that changes in each line and the last inequality holds for $s$ sufficiently large since $\rho<0$.
Combining this estimate with  the second line of \eqref{system_n=2k} we obtain
$$(\ln Z)_s\leq 2k\left(1-\frac{1}{2}\gamma\right) +C e^{\frac{\rho}{2(k-1)\theta}s}.$$
Integrating and using that $e^{\frac{\rho}{2\theta}s}$ is  uniformly bounded for $s\geq 0$ (since $\rho<0$) we obtain
$$ Z(s)\leq C e^{2k\left(1-\frac{1}{2}\gamma\right) s}=Ce^{-\frac{k \rho}{\theta}s}.$$
%{\textcolor{magenta}{cambiar potencia en la última igualdad a $Ce^{-\frac{k\rho}{\theta}s}$}}\\
Hence $Z(s)\sim e^{-\frac{\rho}{2\theta}s} \hbox{ as } s\to \infty,$ or equivalently
$$u(x)=\text{O}(|x|^{-2-\frac{\rho}{\theta}})\hbox{ as } |x|\to \infty.$$
%{\textcolor{magenta}{cambiar potencia de decaimiento a $u(x)=\text{O}(|x|^{-2-\frac{\rho}{\theta}})$}}\\

\subsubsection{$k$-Yamabe steady soliton $\rho=0$.}

In this case we use Equation \eqref{relxzrho=0} to observe that there are constants  $c, C$ such that
$$0< c\leq Z(2-X^{\frac{1}{k}})^{k-1}(s)\leq C, \hbox{ for  $s$ sufficiently large}.$$
From the second equation  \eqref{system_n=2k} we have
$$ c\leq Z^{-\frac{k-2}{k-1}} Z_s\leq C.$$
This implies $cs\leq Z^{\frac{1}{k-1}}\leq Cs,$ or equivalently,
$$u(x)=\text{O}\left( \frac{(\ln |x|)^{\frac{k-1}{k}}}{|x|^2}\right) \hbox{ as } |x|\to \infty.$$

\subsubsection{$k$-Yamabe shrinking soliton $\rho>0$.}
In this case  we have that $\lim_{s\to \infty} (X, Z)=(X_\infty, 0)$ where $2^k=X_B<X_\infty\leq \min\{\gamma^k, X_A\}.$ We also have that $X\leq X_\infty$, which from the second equation in \eqref{system_n=2k} implies
$$Z\geq Ce^{2kd s},$$
where $d=\left(1-\frac{1}{2}X_\infty^{\frac{1}{k}}\right)<0.$ 

 On the other hand, for $s$ large enough, we have that  there is a  $\delta>0$ small such that  $$ X_B^{1+\tfrac{1}{k}}+2\delta =2+2\delta<X^\frac{1}{k}(s).$$ Then  the second equation in \eqref{system_n=2k} implies
 $$Z(s)\leq C e^{-2k \delta s}. $$ 
To finish our estimate we need to separate into two cases  $X_\infty=X_A=4$ and  $X_\infty \ne X_A.$

\noindent{\bf Case $X_\infty\ne  X_A$:} In this  case,  the first equation of  \eqref{system_n=2k} implies (for a different positive constant C)
$$X_s\leq C e^{-2k \delta s}.$$ 
Integrating between $s$ and $\infty$ this implies
$$ X_\infty-X(s)\leq \frac{C}{2k\delta} e^{-2k \delta s}.$$
Combining this estimate with the second equation of \eqref{system_n=2k} 
we obtain 
$$(\ln Z)_s\leq 2k \left(1-\frac{X_\infty}{2}\right)+ C e^{-2k \delta s}.$$
This implies $$Z\leq Ce^{2kd s},$$
with $d=\left(1-\frac{1}{2}X_\infty^{\frac{1}{k}}\right)<0.$ 
Equivalently
$$u (x)=\text{O}(|x|^{-2+2d})\hbox{ as } |x|\to \infty.$$
%with $d=\left(1-\frac{1}{2}X_\infty^{\frac{1}{k}}\right)<0.$ {\textcolor{magenta}{esta expresión se repite tres veces...}}\\

\noindent{\bf Case $X_\infty=X_A=4^{k}$:}
Computing 
$$\left[ \left(1-\frac{X^{\frac{1}{k}}}{4}\right)^k\right]_s=-c_{n,k}\beta ^k X^{\frac{1}{k}-1}Z \left(\gamma-X^{\frac{1}{k}}\right) \geq- Ce^{-2k\delta s}.$$
Integrating between $s$ and $\infty$ this implies
$$-  \left(1-\frac{X^{\frac{1}{k}}}{4}\right)^k\geq - \frac{C}{2k\delta} e^{-2k \delta s}.$$
Equivalently
$$  \left(1-\frac{X^{\frac{1}{k}}}{2}\right)\leq \frac{C}{ka} e^{-2 \delta s}-1.$$
As before we have  from  the second equation of \eqref{system_n=2k} that
$$(\ln Z)_s\leq -2k + C e^{-2 \delta s}.$$
Integrating this implies that $Z\sim e^{-2ks}$ as $s\to \infty$ or equivalently,
$$u (x)=\text{O}(|x|^{-4})\hbox{ as } |x|\to \infty,$$
 and $4= -2+2d$
with $d=\left(1-\frac{1}{2}X_\infty^{\frac{1}{k}}\right)=-1$.

This concludes the proof of Theorem \ref{classification2}.

\end{document}